\documentclass[10pt]{amsart}
\usepackage[active]{srcltx}
\usepackage[all
]{xy} \xyoption{poly}

\usepackage{xspace}
\usepackage{enumitem}
\usepackage[mathscr]{eucal}
\usepackage{amssymb}
\usepackage{a4wide}
\usepackage{xcolor}
\usepackage{ifthen}
\usepackage{amscd}
\usepackage[hyperindex,bookmarksnumbered,
plainpages,backref]{hyperref}

\newcommand{\g}{\mathfrak{g}}
\newcommand{\pp}{\mathfrak{p}}

\newcommand{\n}{\mathfrak{n}}
\newcommand{\gh}{\hat{\mathfrak{g}}}

\def\Z{\mathbb Z}
\def\C{\mathbb C}

\newcommand{\J}{{J}}

\newcommand{\Jhalf}{\J\text{-mod}_{\frac 12}}
\newcommand{\JJ}{\J\text{-mod}_{1}}
\newcommand{\ghat}{\hat\g}
\newcommand{\gm}{\hat\g\text{-mod}_{1}}
\newcommand{\gmnada}{\hat\g\text{-mod}}
\newcommand{\gmi}{\hat\g\text{-mod}_{i}}
\newcommand{\gt}{\hat\g\text{-mod}^{\,t}_{1}}

\newcommand{\ggm}{\g\text{-mod}_{1}}
\newcommand{\gmh}{\g\text{-mod}_{\frac 12}}
\newcommand{\gmhh}{\hat{\g}\text{-mod}_{\frac 12}}
\newcommand{\so}{\mathfrak{so}}
\newcommand{\ssl}{\mathfrak{sl}}
\newcommand{\psl}{\mathfrak{psl}}
\newcommand{\gl}{\mathfrak{gl}}
\newcommand{\po}{\mathfrak{po}}
\newcommand{\spo}{\mathfrak{spo}}
\newcommand{\oo}{\mathfrak{o}}
\newcommand{\h}{\ensuremath{\mathfrak{h}}}
\newcommand{\End}{\operatorname{End}}
\newcommand{\Hom}{\operatorname{Hom}}
\newcommand{\Der}{\operatorname{Der}}
\newcommand{\Ext}{\operatorname{Ext}}
\newcommand{\im}{\operatorname{Im}}
\newcommand{\Ker}{\operatorname{Ker}}
\newcommand{\rad}{\operatorname{rad}}
\newcommand{\ad}{\operatorname{ad}}

\newtheorem{thm}{Theorem}[section]
\newtheorem{prop}[thm]{Proposition}
\newtheorem{lem}[thm]{Lemma}
\newtheorem{rem}[thm]{Remark}
\newtheorem{Cor}[thm]{Corollary}

 \newcommand{\p}{\rho}

\makeindex

\begin{document}

\title{Representations of simple Jordan superalgebras}
\author{Iryna Kashuba and Vera Serganova} 

\ 
\begin{abstract} This paper completes description of categories of representations of finite-dimensional simple 
unital Jordan superalgebras over algebraically closed field of characteristic zero.

\end{abstract}

\maketitle
\section{Introduction}
The first appearance of Jordan superalgebras goes back to the late 70-s, \cite{Kac1}, \cite{Kan}, \cite{Kap}. 
Recall that a $\Z_2$-graded algebra $J=J_{\bar 0}\oplus J_{\bar 1}$ over a field ${\C}$ is called a Jordan superalgebra
if it satisfies the graded identities:
$$
\begin{array}{c}
a\cdot b=(-1)^{|a||b|}a\cdot b,\\
((a\cdot b)\cdot c)\cdot d+(-1)^{|b||c|+|b||d|+|c||d|}((a\cdot d)\cdot c)\cdot b+(-1)^{|a||b|+|a||c|+|a||d|+|c||d|}((b\cdot d)\cdot c)\cdot a=\\ 
=(a\cdot b)\cdot (c\cdot d)+(-1)^{|b||c|}(a\cdot c)\cdot (b\cdot d)+(-1)^{|d|(b+c)}(a\cdot d)\cdot (b\cdot c),
\end{array}
$$
where $a,b,c,d\in J$ and $|a|=i$ if $a\in J_{\bar i}$. The subspace $J_{\bar 0}$ is a Jordan subalgebra of $J$, while $J_{\bar 1}$ is a Jordan bimodule over
$J_{\bar 0}$, they are referred as the even and the odd 
parts of $J$, respectively.

As in the case of Jordan algebras a lot of examples of Jordan superalgebras come from associative superalgebras, or associative superalgebras with superinvolutions. 
Let $A=A_{\bar 0}\oplus A_{\bar 1}$ be {an} associative superalgebra with product $ab$ then
\begin{equation}
a\cdot b=\frac12(ab+(-1)^{|a||b|}ba).
\end{equation}
is the Jordan product on $A$. The corresponding Jordan superalgebra is usually denoted by $A^+$.
Furthermore, if $\star$ is a superinvolution on $A$, then 
$H(A,\star)=\{a\in A\,|\, a^{\star}=a\}$ is a Jordan superalgebra with respect to the product  $a\cdot b$.

The classification of simple finite-dimensional Jordan superalgebras over a field $\mathbb C$ of characteristic zero was obtained in \cite{Kac1} and then completed in \cite{Kan}. 
Then main tool used in both papers was the seminal Tits-Kantor-Koecher (TKK) construction, which associates to a Jordan 
superalgebra $J$ a certain Lie superalgebra $Lie(J)$.
Let us recall this classification; {we use notations from \cite{ZM}}. There are four series of so called Hermitian superalgebras related to the matrix superalgebra
$M_{m,n}:=\operatorname{End}(\mathbb C^{(m|n)})$: $M_{m,n}^{+}$, $m,n\geq 1$, $Q^{+}(n)$, $n\geq 2$, $Osp_{m,2n}$, $m, n \geq 1$ and $JP(n)$, $n\geq 2$;
the Kantor series $Kan(n)$, $n\geq 2$, exceptional superalgebras introduced in \cite{Kan}; a one-parameter family of $4$-dimensional Jordan superalgebras $D_t$, 
$t\in \C$; the Jordan superalgebra $J(V,f)$ of a bilinear form $f$ and, in addition, {the $3$-dimensional 
non-unital Kaplansky superalgebra $K_3$ and the exceptional $10$-dimensional superalgebra  $K_{10}$ introduced by V. Kac in \cite{Kac1}}.

A superspace $V=V_{\bar 0}\oplus V_{\bar 1}$ with the linear map $\beta:J\otimes V\to V$ is a (super)bimodule over a Jordan superalgebra $J$ if
  $J(V):=J\oplus V$ with the product $\cdot$ on $J$ extended by
$$ v\cdot w=0,\,\,a\cdot v=v\cdot a=\beta(a\otimes v)\,\,\text{for}\,\,v,w\in V,\,a\in J$$
  is a Jordan superalgebra.  The category of finite-dimensional $J$-bimodules will be denoted by $J$-mod. Furthermore if
$J$ is a unital superalgebra the category $J$-mod decomposes into the direct sum of three subcategories
\begin{equation}\label{module-splitting}
J\text{-mod}=J\text{-mod}_0\oplus J\text{-mod}_{\frac12}\oplus J\text{-mod}_1
\end{equation} 
according to the action of the identity element $e\in J$, see \cite{ZM2}. The category $J\text{-mod}_0$ consists of
  trivial bimodules only and is not very interesting.
The category of special (or one-sided) $J$-modules, 
$J\text{-mod}_{\frac12}$, consists of $J$-bimodules on which $e\in J$ acts as $\frac12\operatorname{id}$. Finally, the last category consists of bimodules on
which $e$ acts as $\operatorname{id}$, they are called unital bimodules.  
For the  categories of special and unital bimodules one may introduce the corresponding associative universal enveloping algebras
characterized by the property that the categories of their representations are isomorphic to the categories 
$J\text{-mod}_{\frac12}$ and $J\text{-mod}_1$.

 The classification of  bimodules for simple Jordan superalgebras was started in \cite{Sh1} and \cite{Sh2} where unital irreducible bimodules were studied for the exceptional superalgebras $K_{10}$ and $Kan(n)$ respectively. The method used in these papers was to apply the TKK-construction to bimodules, i.e. to associate to any unital Jordan $J$-bimodule a certain graded $Lie(J)$-module. However the answer for $Kan(n)$ was not complete, since in order to describe $J$-mod$_1$ one has to consider modules over the universal central extension
 $\widehat{Lie(J)}$ instead of $Lie(J)$, this was noticed in \cite{ZM4}. In  \cite{MSZ}, \cite{ZM} the
 coordinatization theorem was proved and classical methods from Jordan theory were applied to classify representations of Hermitian superalgebras. In \cite{ZM2} using the universal enveloping algebras authors deduced the problem of describing bimodules over 
 Jordan superalgebra to associative ones. Finally Lie theory proved to be very useful, as already was mentioned the TKK functors can be extended to representations of $J$ and $Lie(J)$ \cite{ZM}, \cite{ZM4}. Observe that the TKK method can only be used in characteristic zero.
 
 In \cite{ZM}, \cite{ZM2}, \cite{ZM3}, \cite{MSZ}, \cite{Tr}, \cite{ShO} finite-dimensional irreducible modules were classified for all simple
 Jordan superalgebras. Moreover it was shown that both categories $J\text{-mod}_{\frac12}$ and $J\text{-mod}_1$ are completely reducible for all simple 
Jordan superalgebras except $JP(2)$, $Kan(n)$, $M^+_{1,1}$, $D_t$ and  superalgebras of bilinear forms. 
The series $D_t$ for $t\neq\pm 1$ was studied in \cite{ZM3}, the authors showed that all special bimodules are completely reducible and
  unital bimodules are completely reducible if $t\neq -\frac{m}{m+2},-\frac{m+2}{m}$ for some $m\in \Z_{>0}$. In the latter case all indecomposable unital
  bimodules were classified in \cite{ZM3}.
 For $t=\pm 1$ we have $D_{-1}\simeq M^+_{1,1}$, and $D_1$ is isomorphic to the Jordan superalgebra of a bilinear form. We study these cases in the present paper.

 We will describe the categories $J\text{-mod}_{\frac12}$ and $J\text{-mod}_1$ when $J$ is one of the following algebras: $JP(2)$, $Kan(n)$, $M^+_{1,1}$ and superalgebras of bilinear form. Our main tool is the functors $Lie$ and $Jor$ between categories 
\begin{equation}\label{Jor-Lie} J\text{-mod}_{\frac12}\leftrightarrow \gmhh \qquad {\rm and} \qquad J\text{-mod}_1\leftrightarrow \gm
\end{equation}
where $\ghat$ is the universal central extension of $\g=Lie(J)$,  $\gm$ is the category of $\ghat$-modules admitting a short grading 
$M=M[-1]\oplus M[0]\oplus M[1]$, while $\gmhh$ the category of $\ghat$-modules admitting a very short grading 
$M=M[-1/2]\oplus M[1/2]$. For the latter pair the functors $Lie$ and $Jor$ establish the equivalence of categories, in the former case the categories 
$J\text{-mod}_1$ and $\gm$ are not equivalent due to the fact that $\gm$ contains the trivial module. 
  More precisely, the splitting \eqref{module-splitting}  $J\text{-mod}_0\oplus J\text{-mod}_1$ can not be lifted to the Lie algebra $\ghat$ since some
  $\ghat$-modules in $\gm$ have non-trivial extensions with the trivial module.

In all non-semisimple cases considered in this paper $\ghat\neq \g$. This has two consequences. There are more irreducible representations
  with non-trivial central charge and there are self extensions on which the center does not act diagonally. In particular, the categories $\gmhh$ and $\gm$ do
  not have enough projective objects and we have to consider the chain of subcategories defined by restriction of the nilpotency degree of central elements.

The paper is organized as follows.
  In section $2$ we recall the Tits-Kantor-Koecher construction, introduce functors $Jor$ and $Lie$ between the categories in \eqref{Jor-Lie} and discuss
  their properties. Section $3$ contains some miscellaneous facts on ext quivers of the categories and Lie cohomology which we use in the rest of the paper.
  In Sections $4$-$7$ we study $\gm$ and $\gmhh$ for $\g=Lie(J)$ with  $J$ equal to $JP(2)$, $Kan(n)$, $n\geq 2$, $M^{+}_{1,1}$ and the Jordan superalgebra of a bilinear form respectively.

We will use several different gradings on a Lie superalgebra $\g$ and fix notations here to avoid the confusion. The $\Z_2$-grading
will be denoted as $\g=\g_{\bar{0}}\oplus\g_{\bar{1}}$. The short $\Z$-grading corresponding to the Tits-Kantor-Koecher construction
will be denoted as $\g=\g[-1]\oplus\g[0]\oplus\g[-1]$. We would like to point out here that this grading is not compatible with the $\mathbb Z_2$-grading.
Finally some superalgebras have another grading consistent with the superalgebra grading, which will be denoted as 
$\g=\g_{-2}\oplus\g_{-1}\oplus\dots\oplus\g_{l}$.

\section{TKK construction for (super)algebras and their representations}
The Tits-Kantor-Koecher construction was introduced independently in 
\cite{Tits}, \cite{Kan}, \cite{Koecher}. We recall it below. For superalgebras it works in the same way as for algebras.

A {\em short grading} of an (super)algebra $\g$ is a $\Z$-grading of the
form $\g=\g[-1]\oplus\g[0]\oplus\g[1]$. Let $P$ be the commutative
bilinear map on a Jordan superalgebra $\J$ defined by $\,P(x,y)=x\cdot y$. Then we
associate to $\J$ a vector space $\g=Lie(\J)$ with short grading
$\g=\g[-1]\oplus\g[0]\oplus\g[1]$ in the following way. We 
put $\g[1]=\J$, $\g[0]=\langle L_a,[L_a,L_b]\,|\,a,b\in\J\rangle$, where $L_a$ 
denotes the operator of left multiplication in $\J$,
and $\g[-1]=\langle P,[L_a,P]\,|\,a\in\J\rangle$ with the following 
bracket
\begin{itemize}
\item $[x,y]=0$ for $x,\,y\in \g[1]$ or $x,\,y\in \g[-1]$; \item
$[L,x]=L(x)$ for $x\in\g[1]$, $L\in\g[0]$; \item
$[B,x](y)=B(x,y)$ for $B\in\g[-1]$ and $x,y\in\g[1]$; \item
$[L,B](x,y)=L(B(x,y))-(-1)^{|L||B|}B(L(x),y) +(-1)^{|x||y|}B(x,L(y))$ for $B\in\g[-1]$,
$L\in\g[0]$, $x,y\in\g[1].$ 
\end{itemize}
Then $Lie(\J)$ is a Lie superalgebra. Note that by construction $Lie(\J)$ is generated as a Lie superalgebra
by $Lie(\J)_1\oplus Lie(\J)_{-1}$. 

Let $\g=\g[-1]\oplus\g[0]\oplus\g[1]$ be a $\Z$-graded Lie superalgebra and let $f\in\g[-1]$ be even element 
of $\g$ ($f\in\g_{\bar 0}$),
then $\Z_2$-graded space $\g[1]=:Jor(\g)$ is a Jordan superalgebra with respect to the product
\begin{equation}
x\cdot y=\left[[f,x],y\right] \qquad x,y\in\g[1].
\end{equation}

A {\em short subalgebra} of a Lie superalgebra $\g$ is an $\ssl_2$
subalgebra spanned by elements $e,h,f$, satisfying $[e,f]=h, [h,e]=e, [h,f]=-f$, such that the eigenspace
decomposition of $ad\,h$ defines a short grading on $\g$. Consider a
Jordan superalgebra $\J$ with unit element $e$. Then  $e$, $h_{\J}=L_e$ and $f_{\J}=P$ span
a short subalgebra $\alpha_J\subset Lie(\J)$. A
$\Z$-graded Lie superalgebra $\g=\g[-1]\oplus\g[0]\oplus\g[1]$ is called
{\em minimal} if any non-trivial ideal $I$ of $\g$ intersects
$\g[-1]$ non-trivially, i.e. $I\cap \g[-1]$ is neither $0$ nor
$\g[-1]$. Then $Jor$ and $Lie$ establish a bijection between Jordan unital superalgebras
and minimal Lie superalgebras with short subalgebras, \cite{CanK}. Furthermore, a unital
Jordan superalgebra $J$ is simple if and only of $Lie(J)$ is a simple Lie superalgebra.

Let $\J$ be a Jordan superalgebra and $\g=Lie(\J)$. By $\hat{\g}$ we denote the universal central extension of $\g$.
Note that the injective homomorphism $\alpha_J\hookrightarrow \g$ can be lifted
to the injective homomorphism $\alpha_J\hookrightarrow \hat{\g}$ since all finite-dimensional representations of $\alpha_J$ are completely reducible.
In particular, $\hat{\g}$ also has a short grading, the center of $\hat{\g}$ is in $\hat{\g}[0]$, and $\hat{\g}[\pm 1]={\g}[\pm 1]$.

Let $\gmhh$ denote the category of finite-dimensional $\hat{\g}$-modules $V$ over
$\hat{\g}$ such that $h\in\alpha_J$ acts on $V$ with eigenvalues $\pm\frac{1}{2}$ and hence induces the grading $V=V[-\frac{1}{2}]\oplus V[\frac{1}{2}]$. {In non-graded case functors $Jor$ and $Lie$ between $\gmhh$ and $\Jhalf$ were introduced in \cite{KS2}. 
The super case is analogous.}
 Define an $\J$-action on
$V[\frac 12]$ by the formula
$$X\circ v=Xfv=[X,f] v \text{\ for\  any\ } X\in\J, v\in V.$$
Then for any $Y\in\J$
$$X\circ (Y\circ v)+(-1)^{|X||Y|}Y\circ (X\circ v)=(XfY+(-1)^{|X||Y|}YfX)fv.$$
On the other hand,
$$(X\circ Y)\circ v=\frac{1}{2}((Xf-fX)Y-(-1)^{|X||Y|}Y(Xf-fX))fv=\frac{1}{2}(XfY+(-1)^{|X||Y|}YfX)fv.$$
Therefore $V[\frac{1}{2}]$ is a special $\J$-module. Set
$Jor(V):=V[\frac{1}{2}]$. Then
$Jor:\gmhh\to\J-$mod$_{\frac{1}{2}}$  is an exact functor between abelian categories.

Next we construct the inverse functor $Lie:\Jhalf\to
\gmhh$. Assume that $M$ is a special $\J$-module. Let
$V=M\oplus M$, for any $X\in\hat{\g}[1]=\J$, $Z=\frac12
[f,[f,Y]]\in\hat{\g}[-1]$, where $Y\in\hat\g[1]=\J$ and
$(m_1,m_2)\in V$ set
$$X(m_1,m_2)=(0,X\circ m_1),\quad Z(m_1,m_2)=(Y\circ m_2,0).$$
Let $\mathfrak h$ be the Lie subalgebra of End $V$
generated by $\hat\g[\pm 1]$. Note that
$$[X,Z](m_1,m_2)=((-1)^{|X||Y|}Y\circ (X\circ m_1),X\circ (Y\circ m_2)).$$
If $A\in\hat\g[1]$, then
$$
\begin{array}{r}
[[X,Z],A](m_1,m_2)=(0, X\circ (Y\circ (A\circ m_1))+(-1)^{|X||Y|+|X||A|+|A||Y|}A\circ (Y\circ
(X\circ m_1)))=\\=(0,((X\cdot Y)\cdot A-(-1)^{|X||Y|}Y\cdot (X\cdot 
A)+X\cdot (Y\cdot A))\circ m_1).\end{array}$$
Similarly if $C=\frac12[f,[f,B]]$ for some $B\in\hat\g[1]$, then
$$
\begin{array}{r} [[X,Z],C](m_1,m_2)=(X\circ (Y\circ (B\circ m_2))+(-1)^{|X||Y|+|X||B|+|B||Y|}B\circ (Y\circ
(X\circ m_2)),0)=\\=(((X\cdot Y)\cdot B-(-1)^{|X||Y|}Y\cdot (X\cdot 
B)+X\cdot (Y\cdot 
 B))\circ m_1,0).\end{array}$$
Let $\rho:\J\to\End(M)$ denote the homomorphism of Jordan
superalgebras corresponding to the structure of the special $\J$-module on $M$, it induces the epimorphism 
$Lie(\rho): \g\to Lie(\rho(J))$, see Theorem 5.15 in \cite{CanK}. The above calculation shows that $Jor(\mathfrak h)=\rho(\J)$. 
By construction of $Lie$ we have the exact sequence
$$
0\to Z({\mathfrak h})\to {\mathfrak h}\to Lie(Jor({\mathfrak h}))\to 0.
$$
Then $Lie(\rho)$ can be lifted to an epimorphism $\hat\g\to \mathfrak
h$. The latter morphism defines a structure of $\hat\g$-module on
$V$. We put $Lie(M):=V$.

\begin{prop}\label{equivalence_of_categories} The functors $Lie$ and $Jor$ define an
equivalence of the categories $\Jhalf$ and $\gmhh$.
\end{prop}
\begin{proof}One has to check $Lie(Jor(V))\simeq V$ and
$Jor(Lie(M))\simeq M$. Both are straightforward.
\end{proof}

Let $\gm$ denote the category of $\hat{\g}$-modules $N$  such that
the action of $\alpha_J$ induces a short grading on $N$, recall that $\JJ$ is the category of unital $\J$-modules. In \cite{KS} the two functors
$$Jor:\gm\to\JJ,\quad Lie:\JJ\to\gm$$ were constructed for Jordan algebra $J$. Analogously, one define these functors in the supercase. 
Namely, if $N\in\gm$, then $N=N[1]\oplus N[0]\oplus N[-1]$. We set $Jor(N):=N[1]$ with action of $\J=\g[1]=\hat\g[1]$ given by
$$x(m)=[f,x]m,\quad x\in\J=\g[1],\ m\in N[1].$$ It is clear that $Jor$ is an exact functor.

Let $M\in\JJ$. Consider the associated null split extension
$\J\oplus M$. Let $\mathcal A=Lie (\J\oplus M)$. Then we have an exact sequence of Lie superalgebras
\begin{equation}\label{exact1}
0\to N\to\mathcal A\xrightarrow{\pi} \g\to 0,
\end{equation}
where $N$ is an abelian Lie superalgebra and $N[1]=M$. By Lemma 3.1, \cite{KS} $M$
is $\hat\g[0]$-module.
Now let $\pp=\hat\g[0]\oplus \g[1]$ and we extend the above  $\hat\g_0$-module structure on $M$ to a $\pp$-module structure
by setting $\g[1] M=0$. Finally we define $Lie(M)$ to be the maximal quotient in $\Gamma (M)=U(\hat\g)\otimes_{U(\pp)}M$ which 
belongs to $\gm$. 

\begin{prop}\label{adjoint}\cite{KS} Functors $Jor$ and $Lie$ have the following properties
\begin{itemize}
\item Let $M\in\gm$ and $K\in\JJ$ $$\operatorname{Hom}_{\hat\g}(Lie(M),K)\simeq \operatorname{Hom}_{\J}(M,Jor(K)),$$
\item If $P$ is a projective module in $\JJ$, then $Lie(P)$ is a projective module in $\gm$.
\item $Jor\circ Lie$ is isomorphic to the identity functor in $\JJ$.
\item  Let $P$ be a projective module in $\gm$ such that $\hat{\g}P=P$. Then $Jor(P)$ is projective in $\JJ$.
\item  Let $L$ be a simple non-trivial module in $\gm$. Then $Jor(L)$ is simple in $\JJ$.
\end{itemize}
\end{prop}

\begin{rem} Note that the correspondence $J\mapsto Lie(J)$ does not define a functor from the category of Jordan superalgebras to the category of Lie superalgebras with short $\ssl(2)$-subalgebra.
    In construction of our functors $Jor$ and $Lie$ we use the following property of TKK construction proven in \cite{CanK}, Section $5$.
    An epimorphism $J\to J'$ of Jordan superalgebras induces the epimorphism $Lie(J)\to Lie(J')$. One can think about analogy with Lie groups and Lie algebras.
    There is more than one Lie group with given Lie algebra. Pushing this analogy further, $\hat\g$ plays the role of a simply connected Lie group. 
  \end{rem}

Let $Z$ denote the center of $\hat{\g}$. For every $\chi\in Z^*$ we denote by $\gm^{\,\chi}$ and $\gmhh^{\chi}$ the full subcategories of 
$\gm$ and $\gmhh$ respectively consisting of the modules annihilated by $(z-\chi(z))^N$ for sufficiently large $N$. We have the decompositions
\begin{equation}\label{subcategory-t}
\gm=\bigoplus_{ {\chi\in Z^*}} \gm^{\,\chi},\qquad \gmhh=\bigoplus_{{\chi\in Z^*}} \gmhh^{\chi}.
\end{equation}

We define $ \Jhalf^{\chi}$ (resp., $\JJ^{\,\chi}$) the full subcategory
  of $\Jhalf$ (resp., $\JJ$) consisting of objects lying in the  image of $\gmhh^{\,\chi}$ (resp., $\gm^{\,\chi}$) under $Jor$.
  It is easy to see that $Jor$ is a full functor. Therefore \eqref{subcategory-t} provides
the decompositions
\begin{equation}\label{jorsubcategory-t}
\JJ=\bigoplus_{ {\chi\in Z^*}} \JJ^{\,\chi},\qquad \Jhalf=\bigoplus_{{\chi\in Z^*}} \Jhalf^{\chi}.
\end{equation}
\begin{rem}\label{corr} Note that $Jor: \gmhh^{\chi}\to  \Jhalf^{\chi}$ is an equivalence of categories. If $\chi\neq 0$, then by Proposition~\ref{adjoint}
 $Jor$ establishes a bijection
between isomorphism classes of simple objects in $\gm^{\,\chi}$ and  $\JJ^{\,\chi}$. Hence in this case it also defines an equivalence of categories.\end{rem}

Furthermore, the categories $\gm^{\,\chi}$ and $\gmhh^{\chi}$ have the filtrations
$$F^1(\gmi^{\chi})\subset F^2(\gmi^{\chi})\subset\dots\subset F^m(\gmi^{\chi})\subset\dots, \quad i=1,\frac12, $$
where $F^m(\mathcal C)$ is the full subcategory of $\mathcal C$ consisting of modules annihilated by $(z-\chi)^m$.
Very often the category $\gm^{\,\chi}$ and $\gmhh^\chi$ do not have projectives but $F^m(\gm^{\,\chi})$ and $F^m(\gmhh^{\chi})$ 
always have enough projective objects.

\section{Auxiliary facts}
\subsection{Quiver of abelian category}

Let $\mathcal C$ be an abelian category and $P$ be a projective generator in $\mathcal C$.
It is a well-known fact (see \cite{Gelf-Manin} ex.2 section 2.6) that the functor $\operatorname{Hom}_{\mathcal C}(P,M)$
provides an equivalence of $\mathcal C$ and the category of right modules over the ring
$A=\operatorname{Hom}_{\mathcal C}(P,P)$. In case when every object in $\mathcal C$  has finite length, $\mathcal C$ has
finitely many non-isomorphic simple objects and every simple object 
has a projective cover, one reduces the problem of classifying indecomposable objects in $\mathcal C$ 
to the similar problem for modules over a finite-dimensional algebra $A$(see \cite{Gab,Gab2}).
If $L_1,\dots,L_r$ is
the set of all up to isomorphism simple objects in $\mathcal C$
and $P_1,\dots,P_r$ are their projective covers, then $A$ is a
pointed algebra which is usually realized as the path algebra of a
certain quiver $Q$ with relations. The vertices of $Q$ correspond
to simple (resp. projective) modules and the number of arrows from
vertex $i$ to vertex $j$ equals to
$\operatorname{dim}\operatorname{Ext}^1(L_j,L_i)$ (resp.
$\operatorname{dim}\operatorname{Hom}(P_i,\operatorname{rad}P_j/\operatorname{rad}^2 P_j$)).

We apply this approach to the case when $\mathcal C$ is  $\gm^\chi$ (respectively $\JJ^\chi$) and $\gmhh^\chi$ 
(respectively $\Jhalf^\chi$). {There is the following relation between quivers of $\gmi^\chi$ and $J\text{-mod}^{\,\chi}_i$}
 \begin{prop}\label{quivers-relation}
  \begin{enumerate}
\item The  $\Ext$ quivers
corresponding to $\gmhh^{\,\chi}$ and $\Jhalf^{\,\chi}$ coincide.
\item If $\chi\neq 0$ the $\Ext$ quivers 
corresponding to $\gm^{\,\chi}$ and $\JJ^{\,\chi}$ coincide.
\item Let $\chi=0$, $Q'$  (resp. $Q$)
  be the $\Ext$ quiver of the category $\JJ^{\,0}$, (resp $\gm^{\,0}$ ) and 
  $A'$ (resp. $A$) be its corresponding path algebra with relations. Then
  $A'=(1-e_0)A(1-e_0)$, where $e_0$ is the idempotent of the vertex $v_0$ corresponding to 
  the trivial representation.
\end{enumerate}
\end{prop}
\begin{proof} First two items follow from Proposition~\ref{equivalence_of_categories}  and Remark~\ref{corr} respectively.
The last part is proved in Lemma 4.10, \cite{KS} for non-graded case and the proof trivially generalizes to supercase.
  \end{proof}
  \begin{rem} Observe that $Q'$ is obtained from $Q$ by removing the vertex $v_0$  and replacing some paths $v\to v_0\to v'$ by the edge $v\to v'$.
    \end{rem}

\subsection{Relative cohomology and extensions}
Let $\g$ be a superalgebra and $M,N$ be two $\g$-modules. Then the extension group $\Ext^i(M,N)$ can be computed via Lie superalgebra cohomology
$$\Ext^i(M,N)\simeq H^i(\g,\Hom_{\mathbb C}(M,N))$$
see, for example, \cite{F}. Let $\h$ be a subalgebra of $\g$ and $\mathcal C$ be the category of $\g$-modules semisimple over $\h$. Then the extension groups
between objects in $\mathcal C$ are given by relative cohomology groups:
$$\Ext_{\mathcal C}^i(M,N)\simeq H^i(\g,\h;\Hom_{\mathbb C}(M,N)).$$
The relative cohomology groups $H^i(\g,\h;X)$ are the cohomology groups of the cochain complex
$$0\to X\to\Hom_{\h}(\Lambda^1(\g/\h),X)\to \Hom_{\h}(\Lambda^2(\g/\h),X)\to \Hom_{\h}(\Lambda^3(\g/\h),X)\to\dots.$$
We use relative cohomology to compute $\Ext^1(M,N)$ when $M,N$ are finite-dimensional $\g$-modules and $\h$ is a simple Lie algebra. The $1$-cocycle
$\varphi\in \Hom_{\h}(\g/\h,X)$ satisfies the condition
 $$\varphi([g_1,g_2])=g_1(\varphi(g_2))-(-1)^{\bar{g}_1\bar{g}_2}g_1(\varphi(g_2)).$$

 We also going to use the following version of Shapiro's lemma for relative cohomology.
 Let $\pp$ be the subalgebra of $\g$ containing $\h$, $M$ be a $\pp$-modules and $N$ be a $\g$-module, then
\begin{equation}\label{shapiro's_lemma}
 H^i(\g,\h;\Hom_{\mathbb C}(\operatorname{Ind}^\g_\pp M,N))\simeq  H^i(\pp,\h;\Hom_{\mathbb C}(M,N)).
 \end{equation}

 \subsection{Some general statements about representations of Lie superalgebras}
 Let $\g$ be a Lie superalgebra and $\h$ be the Cartan subalgebra of $\g$, i.e. a maximal self-normalizing nilpotent subalgebra.
 Then one has a root decomposition
 $\g=\h\oplus\bigoplus\g_{\alpha}$ where $\g_{\alpha}$ is the generalized eigenspace of the adjoint action of $\h_{\bar 0}$.
 Let $\g$ be a simple Lie superalgebra. Assume that $\h_{\bar 1}=0$. It follows from the classification of simple Lie superalgebras
 that this assumption does not hold only for $\mathfrak{q}(n)$ or $H(2n+1)$.  Then for every root $\alpha$ either
 $(\g_{\alpha})_{\bar 0}=0$ or $(\g_{\alpha})_{\bar 1}=0$. Furthermore, if $Q$ is a root lattice of $\g$, one can define a homomorphism
 $p:Q\to\mathbb Z_2$ such that $p(\alpha)$ equals the parity of $\g_{\alpha}$.
 \begin{lem}\label{weighargument}Assume that $\g$ is simple and  $\h_{\bar 1}=0$.
   If $M$ is an indecomposable finite-dimensional $\hat{\g}$-module, then every generalized weight space of $M$ is either purely even or purely odd.
   Hence for a simple module $L$ we have that  $L$ and $L^{op}$ are not isomorphic and do not belong to the same
    block in the category of finite-dimensional $\hat{\g}$-modules.
  \end{lem}
  \begin{proof} Let $M_\mu$ denote the generalized weight space of weight $\mu$. We have $\g_{\alpha}(M_{\mu})\subset M_{\mu+\alpha}$. Therefore all weights of $M$ belong to $\mu+Q$.
Hence the statement follows from existence of parity homomorphism $p$.
\end{proof}

\begin{lem}\label{selfext} Let $\g$ be a Lie superalgebra with semisimple even part and $M$ be a simple finite-dimensional $\g$-module. Then $\Ext_{\g}^1(M,M)=0$.
Furthermore, if 
$\operatorname{sdim}M=\dim M_{\bar 0}-\dim M_{\bar 1}\neq 0$ then
$\Ext_{\hat g}^1(M,M)=0$.
\end{lem}
\begin{proof}
  Consider a
  short exact sequence of $\g$-modules
  $$0\to M\to\tilde M\to M\to 0.$$
Then  $\tilde M$ is generated by a highest weight vectors of some weight $\lambda$
  with respect to some Borel subalgebra of $\g$.
  Since the action of Cartan subalgebra of $\g_{\bar{0}}$  on $\tilde M$ is semisimple the weight space $\tilde M_\lambda$ is a span of two
  highest weight vectors $v_1,v_2$. Then $\tilde M=U(\g)v_1\oplus U(\g)v_2\simeq M\oplus M$ and the sequence splits. 

  Now we prove the second identity. We have to show that $H^1(\g,\g_{\bar 0}, \End(M))=0$. Let $\varphi$ be a non-trivial one-cocycle. By the previous proof $\varphi$ is
  not identically zero on the center of $\hat\g$. On the other hand $[x,\varphi(z)]=0$ for every $x\in\hat\g$ and the central element $z$. By Schur's lemma
  we have $\varphi(z)$ is the scalar operator. Furthermore, there exists $x\in\g_{\bar 1}$ such that $z=[x,x]$. That implies
  $$\varphi(z)=2[x,\varphi(x)].$$
  That implies  $\operatorname{str}(\varphi(z))=0$. If $\operatorname{sdim} M\neq 0$ we obtain $\varphi(z)=0$. That gives a contradiction.
     \end{proof}
 
\section{Representations of  $JP(2)$}
Superalgebras $JP(n)$ and $P(n)$ both emerge from the associative superalgebra $M_{n,n}$ with the superinvolution 
$$
\left[\begin{array}{cc} A&B\\ C& D\end{array}\right]^*=\left[\begin{array}{cc} D^T&B^T\\ -C^T& A^T\end{array}\right],
$$
namely $JP(n)$ is the Jordan superalgebra of symmetric elements,  while $P(n)$
is the Lie superalgebra of skewsymmetric elements of $(M_{n+n}^+,*)$.  These superalgebras also related to each other via the TKK 
construction $Lie(JP(n))=P(2n-1)$, where
$$
JP(n)=\left\{\left[\begin{array}{cc} A&B\\ C& A^{T}\end{array}\right]\,|\, A,\,B,\,C\in M_{n}(\C),\  B^{T}=B,\ C^{T}=-C\right\}=
\left[\begin{array}{cc} A&0\\ 0& A^{T}\end{array}\right]_{\bar 0}+\left[\begin{array}{cc} 0&B\\ C& 0\end{array}\right]_{\bar 1}
$$
and
$$
P(2n-1)=\left\{\left[\begin{array}{cc} A&B\\ C& -A^{T}\end{array}\right]\,|\, A,\,B,\,C\in M_{2n}(\C),\ tr A=0,\  B^{T}=B,\ C^{T}=-C\right\}.
$$
The short grading on $P(2n-1)$ is defined by element  
$$
h=\sum_{i=1}^n E_{i,i}-E_{i+n,i+n}+E_{i+2n,i+2n}-E_{i+3n,i+3n}
$$
and the short $\ssl(2)$ algebra is given by the elements $h$, $e$, $f$, where
$$
e=\sum_{i=1}^n E_{i,i+n}- E_{3n+i,2n+i}, \qquad f=\sum_{i=1}^n  E_{i+n,i}-E_{2n+i,3n+i}.
$$
Observe that we follow notations in \cite{Kac2} and \cite{ZM} where $P(n)$ is the Lie 
superalgebra of rank $n$. Both $JP(n)$, $n\geq 2$ and $P(n)$, $n\geq 3$ are simple superalgebras. 

Another way to describe $P(n)$ is to consider the
$(n+1|n+1)$-dimensional superspace $V$ equipped with odd symmetric
non-degenerate form $\beta$, i.e., the map $S^2(V)\to\mathbb C^{op}$ which
establishes an isomorphism $V^*\simeq V^{op}$. Then $\tilde P(n)$ is
the Lie superalgebra preserving this form and $P(n)=[\tilde P(n),\tilde P(n)]$.
The following isomorphisms of $\tilde P(n)$-modules are important to us
\begin{equation}\label{isomP}
S^2(V^*)\simeq S^2(V^{op})\simeq\Lambda^2(V), \quad S^2(V)\simeq \ad^{op}.  
\end{equation}
The second isomorphism is given by the formula
\begin{equation}\label{beta}v\otimes w\mapsto X_{v,w},\,\,
X_{v,w}(u):=\beta(w,u)v+(-1)^{|v||w|}\beta(v,u)w\,\,\text{for
  all}\,\,u,v,w\in V.
  \end{equation}

Finally, denote by $\hat P(n)$ the universal central extension of 
$P(n)$, then for $n\geq 4$ $P(n)=\hat P(n)$, while the superalgebra $\hat P(3)$  has a one-dimensional center.

\subsection{Construction of $\hat P(3)$-modules with short grading and very short grading}
When $n\geq 3$ both categories $JP(n)\text{-mod}_{\frac12}$, $JP(n)\text{-mod}_1$ are semi-simple, \cite{ZM} and \cite{ZM2}.
In \cite{ZM2} it was shown that the category $JP(2)-{\rm mod}_{\frac12}$ is isomorphic to the category of finite-dimensional
modules over the associative superalgebra $M_{2,2}(\C[t])$, i.e. there exists a one-parameter family of irreducible special $JP(2)$-modules.  Unital irreducible $JP(2)$-modules were described in \cite{ZM}, 
for each $\alpha\in \C$ there are two non-isomorphic modules $R(\alpha)$ and $S(\alpha)$ and their opposite. Modules $R(\alpha)$ and $S(\alpha)$ are constructed as a subspaces in $M_{2+2}(A)$, where $A$ is a certain Weyl algebra. In this section we define a family $W(t)$, $t\in\C$ of special irreducible  $JP(2)$-modules and provide another realization of unital irreducible modules, namely $S^2(W(t/2))$ and 
$\Lambda^2(W(t/2))$. We also construct the ext quiver for $JP(2)\text{-mod}_{\frac12}$ and $JP(2)\text{-mod}_1$.

Let $\ghat$ be the central extension of the simple Lie superalgebra $P(3)$.
There is a consistent (with $\mathbb Z_2$-grading)  $\mathbb Z$-grading 
$$\ghat=\g_{-2}\oplus\g_{-1}\oplus\g_0\oplus\g_1,$$
where $\g_{-2}$ is a one-dimensional center, $\g_0$ is isomorphic to $\so(6)$ and $\g_{-1}$ is the standard $\so(6)$-module.
Furthermore, $\g_1$ is isomorphic to one of the two irreducible components of $\Lambda^3(\g_{-1})$ 
(the choice of the component gives isomorphic superalgebra). The commutator $\g_{-1}\times\g_{-1}\to\g_{-2}$ is given by 
the $\g_0$-invariant form. 

Fix $z\in\g_{-2}$. In \cite{Vera} a $(4|4)$-dimensional
simple $\ghat$-module $V(t)$ on which $z$ acts by multiplication by $t$, $t\in\mathbb C$ was introduced. Let $V=\C^{4|4}$ and define a representation $\rho_t:\ghat\to\End_{\C}(V)$ by 
$$\rho_t\left[\begin{matrix}A&B\\C&-A^t\end{matrix}\right]:=\left[\begin{matrix}A&B+tC^*\\C&-A^t\end{matrix}\right],\quad\rho_t(z):=t,$$
where $c^*_{ij}=(-1)^\sigma c_{kl}$ for the permutation $\sigma=\{1,2,3,4\}\to\{i,j,k,l\}$.
We denote the corresponding $\ghat$-module by $V(t)$. When $t=0$ this module coincides with the standard
$\ghat$-module. Observe that for any $t$, $s\in\C$, $V(t)\simeq V(s)$ as $\g_0+\g_1$-modules.

\begin{rem} The other realization of $V(t)$ is as follows. Let $\mathcal D(3)$ be the superalgebra of differential operators on $\Lambda(\xi_1,\xi_2,\xi_3)$ 
with the odd generators $\xi_1,\xi_2,\xi_3,d_1,d_2,d_3$ satisfying the relation:
$$[d_i,\xi_j]=\delta_{ij}, [\xi_i,\xi_j]=[d_i,d_j]=0.$$
Observe that $\mathcal D(3)$ is isomorphic to the Clifford algebra.
It is easy to see that the Lie subsuperalgebra of $\mathcal D(3)$ generated by $1,d_i,\xi_j,\xi_i\xi_j,d_id_j, \xi_1\xi_2\xi_3$ is isomorphic to
$\hat\g$.  As follows from the general theory of Clifford superalgebras
$\mathcal D(3)$ has a unique  $(4|4)$-dimensional
simple module  $V(1)=\Lambda(\xi_1,\xi_2,\xi_3)$. Since $\mathcal D(3)$ is generated by $d_i,\xi_j$ as the associative algebra, the restriction of $V(1)$ is a simple
$\ghat$-module.
\end{rem}

Let $\sigma_t$ denote the automorphism of $\ghat$ such that $\sigma_t(x)=t^i x$ for every $x\in \g_i$, then
 $V(t)\simeq V(1)^{\sigma_{t^{-1/2}}}$. Note that $V(1)^{\sigma_{-1}}$ is isomorphic to $V(1)$. Hence
the construction does not depend on a choice of the square root.

Observe also that $V(t)^*$ is isomorphic to $V(-t)^{op}$.

It is easy to see that $V(t)$ admits a very short grading with respect to the action of $h$ thus $V(t)\in\gmhh$. Moreover from the equivalence of categories  $M_{2,2}(\C[t])$-mod, $JP(2)\text{-mod}_{\frac12}$ 
 and $\hat{P(3)}\text{-mod}_{\frac12}$, \cite{ZM2}, and  Proposition~\ref{equivalence_of_categories}, it follows that $V(t)$ together with its opposite exhaust  all possibilities for simple objects in $\hat{P(3)}\text{-mod}_{\frac12}$.

\begin{prop}\label{jordanhalfmodule}
Let $t\in\C$. On $W=\C^{2|2}$ define a representation $\rho_t:JP(2)\to \End_{\C}(W)$ by
$$
\rho_t\left[\begin{matrix}A&B\\C&-A^T\end{matrix}\right]:=\left[\begin{matrix}A&B+tC\\C&-A^T\end{matrix}\right].
$$
Then any irreducible module in $JP(2)-$mod$_{\frac{1}{2}}$ is isomorphic either to $W(t)=(W,\rho_t)$ or $W(t)^{op}$.
\end{prop}
\begin{proof} $V(t)\in \gmhh$, thus it is enough to check that $W(t)=Jor(V(t))$. \end{proof}

{The next theorem follows from the equivalence of categories  $M_{2,2}(\C[t])$-mod and $JP(2)\text{-mod}_{\frac12}$, \cite{ZM2}, we give a proof here for the sake of completeness.}
\begin{thm} 
(a)  Every block in the category $\gmhh$ ($JP(2)\text{-mod}_{\frac12}$) has a unique up to isomorphism simple object. 

(b) The category $\gmhh$ ($JP(2)\text{-mod}_{\frac12}$) is equivalent to the category of finite-dimensional $\mathbb Z_2$-graded representations of
the polynomial ring $\mathbb C[x]$.
\end{thm}
\begin{proof} To prove (a) we just note that
$\Ext^1(V(s), V(t))=\Ext^1(V(s), V(t)^{op})=0$ if $t\neq s$ since the modules have different central charge.
Furthermore, from Lemma~\ref{weighargument} we have $\Ext^1(V(t), V(t)^{op})=0$.

To prove (b) we consider the family $V(x)$ defined as above where $x$ is now a formal parameter. Then $V(x)$ is a module 
over $U(\hat\g)\otimes \mathbb \C[x]$. Let $M$ be a finite-dimensional 
$\mathbb \C[x]$-module. Set $F(M):=V(x)\otimes_{\mathbb \C[x]}M$. Obviously $F(M)$ is a $\hat\g$-module. Moreover, $F$ defines an exact functor from the category of
finite-dimensional $\mathbb Z_2$-graded $\mathbb \C[x]$-modules to the category $\gmhh$. The functor $G:=\Hom_{\g}(V(x), ?)$ is its left adjoint.
The functors $F$ and $G$ provide a bijection between isomorphism classes of simple  objects in both categories and hence establish their equivalence.
\end{proof}

Now we will describe the simple modules in the category $\gm$. Let us consider the decomposition
$$V(t/2)\otimes V(t/2)= S^2V(t/2)\oplus \Lambda^2 V(t/2).$$
Then clearly both $ S^2V(t/2)$ and $\Lambda^2 V(t/2)$ are objects in $\gm$ and have central charge $t$.

\begin{lem}\label{simple} 

(a) If $t\neq 0$, then $ S^2V(t/2)$ and $\Lambda^2 V(t/2)$ are simple. 

(b) If $t=0$ we have the following exact sequences
$$0\to L^+(0)\to S^2(V)\to \C^{op}\to 0,\quad 0\to \C^{op}\to \Lambda^2(V)\to  L^-(0)\to 0,$$
where $L^\pm(0)$ are some simple $\g$-modules.
\end{lem}

\begin{proof} Let us prove (b). The first exact sequence follows from existence of $\g$-invariant odd symmetric form $\beta$ on $V$, \eqref{beta},  
the second is the dualization.
Moreover $L^+(0)^{op}$ is the adjoint representation in $P(3)$, hence simple. But then $L^+(0)$ 
is obviously simple, $L^-(0)$ is simple by duality.

To prove (a) we observe that $ S^2V(t/2)$ is a polynomial deformation of $S^2(V)$. Moreover, for all $t\neq 0$ the corresponding modules are 
related by twisting with an automorphism. Thus, either $ S^2V(t/2)$ is simple or it has a $1$-dimensional quotient. 
But there is no one dimensional module with
non-zero central charge. Hence $ S^2V(t/2)$ is simple. The proof for $\Lambda^2 V(t/2)$ follows by duality.
\end{proof}

For $t\neq 0$ we set $L^+(t)=S^2V(t/2)$, $L^-(t)=\Lambda^2V(t/2)$.
\begin{thm} A simple object in  $\gm$ is isomorphic to one of the following: $L^\pm(t), L^\pm(t)^{op}, \C$ or $\C^{op}$.
\end{thm}
\begin{proof} It follows from Theorem 3.10, \cite{ZM2} that for an arbitrary  $t\in\C$
there are exactly four non-isomorphic simple objects in $\JJ^{\,t}$. Comparing their dimensions one can see that the image of these modules 
via the TKK-constructions is one of  $L^\pm(t)$ or $L^\pm(t)^{op}$. 
Adding the one-dimensional trivial module and its opposite to $\ggm$ we finish the proof.
\end{proof}

Recall that $W(t)$, $t\in \C$ is the irreducible special $JP(2)$-module defined in Lemma~\ref{jordanhalfmodule}. Then $W(t)\otimes W(t)$
has a structure of unital $JP(2)$-module, \cite{J}. As a superspace $W(t)\otimes W(t)=S^2(W(t))\oplus \Lambda^2(W(t))$. 

\begin{Cor} Both $S^2(W(t/2))$, $\Lambda^2(W(t/2))$ are simple $JP(2)$-modules. A simple module in $JP(2)\text{-mod}_1$ is isomorphic to one of the following: $S^2(W(t/2))$, $\Lambda^2(W(t/2))$ and their opposites.
\end{Cor}
\begin{proof} One can easily check that $Jor(L^+(t))= S^2(W({t/2}))$, $Jor(L^-(t))= \Lambda^2(W(t/2))$ for any $t\in\C$. The rest follows from previous theorem and from Proposition~\ref{adjoint}.
\end{proof}

Recall that $\gt$ is the full subcategory of $\gm$ consisting of modules on which $z$ acts with generalized eigenvalue $t$.
Note that if $t,s\neq 0$ then $\gt$ and $\gm^s$ are equivalent, by twist with $\sigma_{t^{1/2}s^{-1/2}}$.

\begin{lem}\label{invariants} Let $t\neq 0$. We have the following isomorphisms of $\g_0$-modules
$$H^0(\g_1,L^-(t))\simeq \Lambda^2(V_{\bar 0})\oplus \C,\quad
  H^0(\g_1,L^+(t))\simeq S^2(V_{\bar 0}),$$

  $$H_0(\g_1,L^-(t))\simeq S^2(V_{\bar 1}),\quad
  H_0(\g_1,L^+(t))\simeq \Lambda^2(V_{\bar 0})\oplus \C.$$

\end{lem}
\begin{rem} Observe that $\g_0\simeq\mathfrak{sl}(4)$ and $V_{\bar 0}$ (resp.,$V_{\bar 1}$) are the standard (resp., costandard) $\g_0$-modules. \end{rem}
\begin{proof} Consider the subalgebra $\g^+:=\g_0\oplus\g_1$. 
  Recall that $V(t)$ is isomorphic to $V$ as a $\g^+$-module. Therefore
  $L^+(t)=S^2(V_{t/2})$ is isomorphic to
  $S^2(V)$ and  $L^-(t)$ is isomorphic to
  $\Lambda^2(V)$ as $\g^+$-modules. Hence the statement follows from
  Lemma \ref{simple}(b).
\end{proof}

Let $\mathfrak p=\g_{-2}\oplus\g_0\oplus\g_1$ and $\C_t$ be the $(0|1)$-dimensional $\mathfrak p$-module with central charge $t$. Consider the induced module
$$K(t):=\operatorname{Ind}_{\mathfrak p}^{\g}\C_t\simeq\operatorname{Coind}_{\mathfrak p}^{\g}\C_t.$$

\begin{prop} The category $\gt$ has two equivalent blocks $\Omega_t^+$ and $\Omega_t^-$. The equivalence of these blocks is established by
the change of parity functor. If $t\neq 0$, then $\Omega_t^+$  has two simple objects $L^+(t)$ and $L^{-}(t)$. The block $\Omega_0^+$ has three simple objects
$\C^{op}$,  $L^+(0)$ and $L^{-}(0)$.  
\end{prop}
\begin{proof} By the weight parity argument, Lemma~\ref{weighargument}, $\Ext^1(L^\pm(t),L^\pm(t)^{op})=0$. For $t=0$  the statement follows from the fact that the 
sequences in
Lemma \ref{simple}  do not split.
It remains to show 
$\Ext^1(L^+(t),L^{-}(t))\neq 0$ if $t\neq 0$. It follows from Lemma ~\ref{invariants}  that
$$\Hom_{\g_0}(\C_t,H_0(\g_1,L^+(t)))=\C,\quad \Hom_{\g_0}(\C_t,H^0(\g_1,L^-(t)))=\C.$$
By Frobenius reciprocity we have a surjection $K(t)\to L^-(t)$ and injection $L^+(t)\to K(t)$. A simple check of dimensions implies the exact sequence
$$0\to L^+(t)\to K(t)\to L^-(t)\to 0$$
and it remains to prove that it does not split. Indeed,
$$\Hom_\g(K(t), L^+(t))=\Hom_{\mathfrak p}(\C_t, L^+(t))=\Hom_{\g_0\oplus\g_{-2}}(\C_t, H^0(\g_1, L^+(t)))=0.$$
\end{proof}

\begin{lem}\label{dual} We have isomorphisms
$$L^+(t)^*\simeq L^-(-t),\quad L^-(t)^*\simeq L^+(-t),\quad K(t)^*\simeq K(-t).$$
\end{lem}
\begin{proof} Follows from the isomorphism $V^*(t/2)\simeq V^{op}(-t/2)$.
\end{proof}

\subsection{Unital modules with non-zero central charge}
\begin{lem}\label{lmext}   
If $t\neq 0$ we have   
\begin{enumerate}
\item $\Ext^1(L^+(t),L^{+}(t))=\Ext^1(L^-(t),L^{-}(t))=\C$;
\item $\Ext^1(L^-(t),L^{+}(t))=\C$;
\item $\Ext^1(L^+(t),L^{-}(t))=0$.
\end{enumerate}
\end{lem}

\begin{proof} For (1) first we show that  $\Ext^1(L^-(t),L^{-}(t))\neq 0$. For this consider a non-trivial self-extension 
$$0\to V(t/2)\to\bar V(t/2)\to V(t/2)\to 0.$$
The action of $z$ on $\bar V(t/2)$ is given by the Jordan blocks of size $2$. Now consider $\Lambda^2\bar V(t/2)$. Then the 
Jordan-Hoelder multiplicities are as 
follows:
$$[\Lambda^2\bar V(t/2):L^-(t)]=3,\quad [\Lambda^2\bar V(t/2):L^+(t)]=1.$$
Moreover, the action of $z$ on $\Lambda^2\bar V(t/2)$ is given by Jordan blocks of size $3$ and $1$. This implies 
that $\Lambda^2\bar V(t/2)$ contains a non-trivial
self-extension of $L^-(t)$. 

Now we show that $\Ext^1(L^-(t),L^-(t))$ is one-dimensional. Indeed, let $\psi: \g\to \End_\C(L^-(t))$ be a cocycle defining
the extension. The cocycle condition implies that $\psi(z)\in  \End_{\hat\g}(L^-(t))=\C$. Therefore if $\dim  \Ext^1(L^-(t),L^{-}(t))>1$,
then there exists a non-trivial cocycle
$\psi$ such that $\psi(z)=0$. Consider the corresponding self-extension
$$0\to L^-(t)\to M\to L^-(t)\to 0.$$
Note that $M^{\g_1+\g_0}$ is isomorphic to $\C_t\oplus \C_t$ as $\g_0+\g_{-2}$-module. Therefore $M$ is a quotient of $K(t)\oplus K(t)$ and hence 
$M\simeq  L^-(t)\oplus L^-(t)$. Thus, the corresponding extension is trivial.
Finally, since $L^-(-t)^*\simeq L^+(t)$, we obtain by duality that  $\Ext^1(L^+(t),L^{+}(t))=\C$.

Next we will prove (2). Consider a non-split extension
$$0\to L^+(t)\to M\to L^-(t)\to 0.$$
Since coinvariants is a right exact functor, there exists a surjection
$H_0(\g_1,M)\rightarrow  H_0(\g_1,L^-(t))$. Hence by Lemma~\ref{invariants}
$\Hom_{\mathfrak p}(M,\C_t)\neq 0$. By the Frobenius
reciprocity we must have a non-zero map
$$\phi:M\to\operatorname{Coind}_{\mathfrak p}^{\g}\C_t=K(t).$$ 
Since the socles of $M$ and $K(t)$ are isomorphic and both modules
have length $2$, $\phi$ is an isomorphism. 
Hence $\Ext^1(L^-(t),L^+(t))$ is one-dimensional.

Finally we will show (3). Assume that there is a non-split exact sequence
$$0\to L^-(t)\to M\to L^+(t)\to 0.$$
Consider the following piece of the long exact sequence
$$\dots\to H^0(\g_1,M)\xrightarrow{r} H^0(\g_1,L^+(t))\xrightarrow{r'}  H^1(\g_1,L^-(t))\to\dots.$$
By Lemma \ref{invariants} we have
$H^0(\g_1,L^+(t))=S^2(V_{\bar 0})$. We use the decomposition of
$L^-(t)$ as an $\g_0=\ssl(4)$-module: 
$$L^-(t)\simeq\C\oplus\Lambda^2(V_{\bar 0})\oplus \ssl(4)\oplus S^2(V_{\bar 1}).$$ 
Since  $H^1(\g_1,L^-(t))$ is a submodule in
$$\g^*_1\otimes L^-(t)=S^2(V_{\bar 1})\otimes (\C\oplus\Lambda^2(V_{\bar 0})\oplus
\ssl(4)\oplus S^2(V_{\bar 1})),$$
we conclude that $H^1(\g_1,L^-(t))$ does not contain an
$\g_0$-submodules, isomorphic to $S^2(V_{\bar 0})$. Since $r$ and $r'$ are morphisms of $\g_0$-modules, $r'=0$.
Thus, we obtain that $r$ is surjective and therefore $M$ is a quotient of the induced module 
$\operatorname{Ind}_{\mathfrak p}^{\g}S^2(V_{\bar 0})$, (here we assume that $z$ acts on $S^2(V_{\bar 0})$ 
as $t$ and $\g_1$ acts by zero). Next consider  an isomorphism of $\g_0$-modules
$$\operatorname{Ind}_{\mathfrak p}^{\g}S^2(V_{\bar 0})\simeq
\Lambda^{\cdot}(\Lambda^2(V_{\bar 1}))\otimes S^2(V_{\bar 0})$$
which implies
$$\Hom_{\g_0}(\operatorname{Ind}_{\mathfrak p}^{\g}S^2(V_{\bar 0}),\C)=\Hom_{\g_0}(\Lambda^{\cdot}(\Lambda^2(V_{\bar 1}), S^2(V_{\bar 1}))=\C.$$
On the other hand, $\Hom_{\g_0}(M,\C)=\C^2$ and we obtain a contradiction.
\end{proof}

\begin{thm}\label{quiver-p(3)-non-zero-charge} If $t\neq 0$, then the category $\Omega_t^+$ is equivalent to the category of nilpotent representations of the quiver
$$
\xymatrix{\bullet \ar@(ul,ur)[]|{\alpha} \ar@<0.5ex>[r]^\beta & \bullet\ar@(ul,ur)[]|{\gamma}}
$$
with relations $\beta\alpha=\gamma\beta$.
\end{thm}
\begin{proof} Consider the subcategories $F^{m}(\gm^t)$ of $\gm^t$ defined in Section $2$.
\begin{lem} Let 
  $K(t)_{(m)}:=\operatorname{Ind}^\g_{\mathfrak p}(\C[z]/((z-t)^m)$ and
  $L^+(t)_{(m)}$ be the indecomposable of length $m$ with all composition factors isomorphic to $L^+(t)$. Then $K(t)_{(m)}$ and  
  $L^+(t)_{(m)}$
 are projective covers of $L^-(t)$ and $L^+(t)$, respectively, in the category $F^{m}(\gm^t)$.
\end{lem}
\begin{proof}
 The projectivity of  $L^+(t)_{(m)}$ follows easily by induction on $m$. Indeed, in the case $m=0$, we have $\Ext^1(L^+(t),L^{-}(t))=0$ 
 and in the only 
non-trivial self-extension of  $L^+(t)$ the action of the center is not semisimple. Then by induction and the long exact sequence we get
$\Ext^1(L^+(t)_{(m)},L^{-}(t))=0$ and the only non-trivial extension $\Ext^1(L^+(t)_{(m)},L^{+}(t))$, the action of the center is given by the Jordan block of length $m+1$.

To prove the projectivity of $K(t)_{(m)}$ we have to show $$\Ext_{(1)}^1(K(t), L^\pm(t))=0$$
where $\Ext_{(1)}$ stand for extension in the category $F^{(1)}(\gm^t)$ and then again proceed by induction as in the previous case.
We recall the exact sequence $$0\to L^+(t)\to K(t)\to L^-(t)\to 0.$$
Consider the corresponding long exact sequences for computing $\Ext_{(1)}^1(K(t), L^\pm(t))$.
For $\Ext_{(1)}^1(K(t), L^-(t))$ we get
$$0=\Ext_{(1)}^1(L^-(t), L^-(t))\to \Ext^1_{(1)}(K(t), L^-(t))\to \Ext_{(1)}^1(L^+(t), L^-(t))=0$$
and for $\Ext_{(1)}^1(K(t), L^+(t))$ we get
$$\begin{array}{c}
0=\Hom(K(t), L^+(t))\to \Hom(L^+(t), L^+(t))\to \Ext_{(1)}^1(L^-(t), L^+(t))\to \vspace{0,1cm}\\
\qquad \qquad \to \Ext_{(1)}^1(K(t), L^+(t))\to \Ext_{(1)}^1(L^+(t), L^+(t))=0, \end{array}$$
$$\Hom(L^+(t), L^+(t))\simeq \Ext_{(1)}^1(L^-(t), L^+(t))=\C.$$
Thus  $\Ext_{(1)}^1(K(t), L^+(t))=0$.
\end{proof}

Finally the relation $\beta\alpha=\gamma\beta$ follows from the calculation of the second and the third terms of the radical filtration 
for $K(t)_{(m)}$ and $L^+(t)_{(m)}$ for the large $m$.
Indeed, 
$$\rad K(t)_{(m)}/\rad^2 K(t)_{(m)}=\rad^2 K(t)_{(m)}/\rad^3 K(t)_{(m)}=L^+(t)\oplus L^-(t),$$
and 
$$\rad L^+(t)_{(m)}/\rad^2 L^+(t)_{(m)}=\rad^2 L^+(t)_{(m)}/\rad^3 L^+(t)_{(m)}= L^+(t).$$
\end{proof}

\subsection{The case of zero central charge}

\begin{lem}\label{extzero} For $t=0$ we have
\begin{enumerate}
\item $\Ext^1(L^+(0),L^+(0))=\Ext^1(L^-(0),L^-(0))=\Ext^1(L^+(0),L^{-}(0))=0$;
\item $\Ext^1(L^-(0),L^+(0))=\C$;
\item $\Ext^1(L^{\pm}(0),\C^{op})=\C$;
\item $\Ext^1(\C^{op}, L^{\pm}(0))=\C$.
\end{enumerate}
\end{lem}
\begin{proof} In view of Lemma  \ref{selfext} we already have that $\Ext^1(L^\pm(0),L^\pm(0))=0$.
  Let us show that $\Ext^1(L^+(0),L^{-}(0))=0$. Recall the proof of Lemma \ref{lmext}(3). By the same argument as in this proof,
  we obtain that if the sequence
  $$0\to L^{-}(0)\to M\to L^+(0)\to 0$$
  does not split then $M$ is a quotient of the induced module 
  $\operatorname{Ind}_{\mathfrak p}^{\g}S^2(V_{\bar 0})$. By (13) Section 4.3 in \cite{Vera} this induced module does not have a simple constituent isomorphic to
  $L^-(0)$. Therefore
  there is no such non-split exact sequence. This completes the proof of (1).

  By Lemma \ref{simple} (b) $\Ext^1(L^{-}(0),\C^{op})\neq 0$ and  $\Ext^1(\C^{op},L^{+}(0),)\neq 0$. To prove that other extensions are not zero, consider the
  Kac module $K^{op}(0)$. We claim that it has the following radical filtration
  $$\begin{array}{cc} K^{op}(0)/\rad K^{op}(0)=\C^{op},& \rad K^{op}(0)/\rad^2 K^{op}(0)=L^-(0),\\
      \rad^2 K^{op}(0)/\rad^3 K^{op}(0)=L^+(0),& 
      \rad^3 K^{op}(0)/\rad^4 K^{op}(0)=\C^{op},\\
      \rad^4 K^{op}(0)=0.
      \end{array}$$
  Indeed, $K^{op}(0)=U(\g_{-1})v$ for a $\g_0$-invariant vector $v$. Moreover,
  $$\Hom_{\g}(K^{op}(0),L^{\pm}(0))=0,$$
  since $(L^\pm(0))^{\g_0}=0$. That proves $K^{op}(0)/\rad K^{op}(0)=\C^{op}$. Furthermore, $\g_1\g_{-1}v=0$, hence the maximal submodule
  $N$ of $K^{op}(0)$ is generated by $\g_{-1}v$. Thus, $N$  is a quotient of the induced module $\operatorname{Ind}_{\mathfrak p}^{\g}\Lambda^2(V_{\bar 1})$ and hence $N$ has a simple cosocle
  isomorphic to $L^-(0)$. That implies $\rad K^{op}(0)/\rad^2 K^{op}(0)=L^-(0)$. Finally the rest follows from the self-duality of $K^{op}(0)$.

By considering different subquotients of length $2$ of $K^{op}(0)$ we obtain non-trivial elements in 
$\Ext^1(\C^{op},L^{-}(0))$, $\Ext^1(L^{-}(0), L^{+}(0))$ and
$\Ext^1(L^{+}(0),\C^{op})$. To finish the proof of Lemma we have to show that all above $\Ext^1$ groups are one-dimensional.

Recall that $L^-(0)\simeq \ad^{op}$. Using the duality and change of parity functor it suffices to check that
$\Ext^1(\C,\ad)$, $\Ext^1(\C,\ad^*)$ and $\Ext^1(\ad^*, \ad)$ are one-dimensional. First we have 
$\Ext^1(\C,\ad)=\operatorname{Der}(\g)/\g=\mathbb C$, see \cite{Kac2}. Next,
$$\dim\Ext^1(\C,\ad^*)\leq\dim\Hom_{\g_0}(\g_{1}\oplus\g_1,\ad^*)=1.$$
Now let us prove that $\dim\Ext^1(\ad^*,\ad)\leq 1$.
The Lie superalgebra $\g$ has a root decomposition with even roots
$$\Delta_{\bar 0}=\{(\pm(\varepsilon_i\pm\varepsilon_j)\,|\, 1\leq i<j\leq 3\},$$
and the odd roots
$$\Delta_{\bar 1}=\{\pm\varepsilon_1,\pm\varepsilon_2,\pm\varepsilon_3,\varepsilon_1+\varepsilon_2+\varepsilon_3,
\varepsilon_1-\varepsilon_2-\varepsilon_3,-\varepsilon_1-\varepsilon_2+\varepsilon_3,-\varepsilon_1+\varepsilon_2-\varepsilon_3\}.$$
Note that the odd roots $\pm\varepsilon_i$ have multiplicity $2$ and the roots $\varepsilon_1+\varepsilon_2+\varepsilon_3$,
$\varepsilon_1-\varepsilon_2-\varepsilon_3,-\varepsilon_1-\varepsilon_2+\varepsilon_3,-\varepsilon_1+\varepsilon_2-\varepsilon_3$ are not invertible.
Let $\Delta^{+}$ (respectively, $\Delta^{-}$) be the set of roots $a\varepsilon_1+b\varepsilon_2+c\varepsilon_3$ such that $a+2b+4c>0$ (respectively, $a+2b+4c<0$).
The decomposition $\Delta=\Delta^+\cup\Delta^-$ defines a triangular decomposition $\g=\n^-\oplus\h\oplus n^+$. Every finite-dimensional simple $\g$-modules
has a unique up to proportionality lowest weight vector. The lowest weight of $\ad$ is $\nu=-\varepsilon_2-\varepsilon_3$ and the lowest weight of $\ad^*$ is
$\lambda=-\varepsilon_1-\varepsilon_2-\varepsilon_3$. Let $M$ be an indecomposable $\g$-module of length $2$ with socle $\ad$ and cosocle $\ad^*$. Then $M$ is generated by the
lowest weight vector of weight $\lambda$. Hence $M$ is a quotient of the Verma module $M(\lambda):=U(\g)\otimes_{U(\h\oplus\n^{-})}\C_\lambda$. Multiplicity of
weight $\nu$ in $M(\lambda)$ equals $2$ since the multiplicity of the simple root $\varepsilon_1$ is $2$. However, $\nu$ appears as a weight of $\ad^*$ as well as a weight of $\ad$, hence $\ad$ appears in
$M(\lambda$ with multiplicity at most one. The proof is complete.
\end{proof}
\begin{thm}\label{quiver-p(3)-zero-charge} The Ext quiver  of the category $\Omega_0^+$ is 
$$
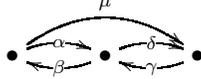
\xymatrix{\bullet\ar@/^1.2pc/[rr]^{\mu} \ar@/^0.4pc/[r]|{\alpha}&\bullet\ar@/^0.4pc/[r]|{\delta}\ar@/^0.4pc/[l]|{\beta}&\bullet\ar@/^0.4pc/[l]|{\gamma}}
$$
Therefore the category $\Omega_0^+$ is equivalent of the category of nilpotent representations of the path algebra of the above quiver modulo some relations.
These relations include $\delta\alpha=\beta\gamma=0$, $\mu\beta\alpha=\delta\gamma\mu$ .
\end{thm}
\begin{rem} We suspect that there is no other relations but this fact is not needed for the description of the corresponding category for the Jordan algebra.
\end{rem}
  \begin{proof} Lemma \ref{extzero} implies that the above quiver is the Ext quiver of $\Omega_0^+$, where the left vertex corresponds to $L^+(0)$, the right vertex to
  $L^-(0)$ and the middle vertex to $\C^{op}$. We have to prove the relations.

  Showing that $\delta\alpha=0$ is equivalent to proving that there is no $\g$-module $R$
  with socle isomorphic to $L^+(0)$ and cosocle isomorphic to $L^-(0)$ with middle layer of the radical filtration  $\C^{op}$. In the proof of Lemma \ref{extzero}
  we constructed a module $M$ of length $2$ with socle $L^+(0)$ and cosocle $L^-(0)$ which is a quotient of the Verma module $M(\lambda)$. Since the multiplicity of
  weight $\nu$ in $M(\lambda)$, $M$ and $R$ is the same and equals $2$, we obtain that $M=M(\lambda)/N$ and $R=M(\lambda)/Q$, where $N$ and $Q$ are maximal submodules
  of $M(\lambda)$ which intersect weight spaces of weights $\lambda$ and $\nu$ trivially. Since $Q+N$ satisfies the same property, maximality of $N$ and $Q$ implies
  $N=Q$.

  Next we show that $\beta\gamma=0$. It suffices to prove that there is no $\g$-module $F$
  with socle isomorphic to $L^-(0)$ and cosocle isomorphic to $L^+(0)$ with middle layer of the radical filtration  $\C^{op}$. 
  Assume that such $F$ exists. Then $zF=0$.
  We have an isomorphism of $\g$-modules
  $$(F/\operatorname{soc} F)^{op}\simeq \g.$$
  Choose a non-zero $v\in F^{\g_0}$. Then by above isomorphism for any $x\in\g_{-1}$ such that $[x,x]\neq 0$ we have $v\in\operatorname{Im}x$. Since $zF=0$
  and $[x,x]=2x^2=cz$, we obtain $xv=0$. Therefore $\g_{-1}v=0$. On the other hand, $\g_1v=0$ as $L^-(0)$ does not have $\g_0$ components isomorphic to $\g_1$. That
  implies $v\in F^{\g}$, that leads to a contradiction.

Finally we show the relation $\mu\beta\alpha=\delta\gamma\mu$. If for the sake of contradiction we assume that this relation does not hold, then there exists
a $\g$-module $T$ with the following radical filtration:
 \begin{equation}\label{filtration-loops}
\begin{array}{c}
L^-(0)\\
\overline{\C^{op} \oplus L^+(0)}\\
\overline{L^-(0)\oplus \C^{op}}\\
\overline{L^+(0)\oplus L^+(0)}
\end{array}
\end{equation}
In particular we have $\rad T=T'\oplus T''$, where $T'$ has cosocle $\C^{op}$ and $T''$ has cosocle $L^+(0)$.
Note that $zT\neq 0$ and $z^2T=0$. This implies that the submodule $zT$ has length $2$ with cosocle $L^-(0)$ and socle $L^+(0)$. Therefore $zT\subset T'$.
On the other hand, $zT''\neq 0$. A contradiction.
\end{proof}

\begin{thm}\label{description-JP(2)} The category $\JJ$ consists of infinite number of equivalent blocks, each block is 
 equivalent to the category of nilpotent representations of the quiver
$$
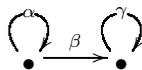
\xymatrix{\bullet \ar@(ul,ur)[]|{\alpha} \ar@<0.5ex>[r]^\beta & \bullet\ar@(ul,ur)[]|{\gamma}}
$$
with relations $\beta\alpha=\gamma\beta$.
\end{thm}
\begin{proof} {It follows immediately by applying Proposition~\ref{quivers-relation} to quivers obtained in Theorem~\ref{quiver-p(3)-non-zero-charge} and Theorem~\ref{quiver-p(3)-zero-charge}  }
\end{proof}
\begin{rem}
This quiver has wild representation type, see (12), Table W in \cite{Han}.
\end{rem}

\section{Representations of $Kan(n)$, $n\geq 2$}
Let $\Lambda(n)$ be the Grassmann superalgebra generated by $n\geq 2$ odd generators $\{\xi_1,\dots,\xi_n\}$
such that $\xi_i\xi_j+\xi_j\xi_i=0$. Define odd superderivations $\frac{\partial}{\partial\xi_i}$, $i=1,\dots,n$ on $\Lambda(n)$\begin{equation}
\frac{\partial}{\partial\xi_i}\  \frac{\partial\xi_j}{\partial\xi_i} =\delta_{ij}, \quad
\frac{\partial(uv)}{\partial\xi_i}=\frac{\partial u}{\partial\xi_i}v+(-1)^{|u|}u\frac{\partial v}{\partial\xi_i}.
\end{equation} 
Then the linear superspace $J_n=\Lambda (n)\oplus\overline{\Lambda (n)}$, 
is a Jordan superalgebra with respect to the product $"\cdot"$
\begin{equation}
f\cdot g=fg \quad f\cdot\overline{g}=\overline{fg}, \quad 
\overline{f}\cdot\overline{g}:=\{f,g\}=(-1)^{|f|}\sum_{i=1}^n \frac{\partial f}{\partial \xi_i}\frac{\partial g}{\partial \xi_i}.
\end{equation}
Here $\overline{\Lambda (n)}$ is a copy of $\Lambda(n)$, $f,g\in \Lambda(n)$, both homogeneous and $\{f,g\}$ is Poisson bracket.
The $\Z_2$-grading of $J_n=(J_n)_{\bar{0}}+(J_n)_{\bar{1}}$ is given by 
$(J_n)_{\bar{0}}=\Lambda(n)_{\bar{0}}+\overline{\Lambda(n)_{\bar{1}}}$ and $(J_n)_{\bar{1}}=\Lambda(n)_{\bar{1}}+\overline{\Lambda(n)_{\bar{0}}}$. The superalgebra $J_n$ is called the Kantor double of the Grassmann Poisson superalgebra and it is simple Jordan superalgebra for any $n\geq 1$.  Observe that $J_1$ is isomorphic to the general linear superalgebra $M^+_{1,1}$ (this superalgebra will be considered in next Section) and 
for $n\geq 2$, $J_n$ is exceptional. 

To determine the TKK construction of $Kan(n)$ we will introduce another set of generators of $J_n$, namely
if $n=2k$ define 
\begin{equation}
\eta_i=\frac{1}{\sqrt{2}}\left(\frac{\partial f}{\partial\eta_i}+\frac{\partial f}{\partial\eta_{k+i}}\right), \quad 
\eta_{i+k}=\frac{1}{\sqrt{2}}\left(\frac{\partial f}{\partial\eta_i}-\frac{\partial f}{\partial\eta_{k+i}}\right), \quad i=1,\dots, k,
\end{equation}
while if $n=2k+1$ add $\eta_0=\frac{1}{\sqrt{2}}\xi_{2k+1}$. The Poisson bracket may be rewritten as
\begin{equation}
\{f,g\}=(-1)^{|f|}\left(\sum_{i=1}^k \left(\frac{\partial f}{\partial\eta_i} \frac{\partial g}{\partial\eta_{i+k}}+\frac{\partial f}{\partial\eta_{i+k}}\frac{\partial g}{\partial\eta_i}\right)+\frac12 \frac{\partial f}{\partial\eta_0} \frac{\partial g}{\partial\eta_0}\right),
\end{equation}
where the last summand only appears for odd $n$.

The Poisson Lie superalgebra $\po(0\,|\,n)$ can be describe as $\Lambda(n)$ endowed with the bracket
$[f,g]=-\{f,g\}$.  Let $\spo(0\,|\,n)=[\po(0\,|\,n),\po(0\,|\,n)]$, then $H(n)=\spo(0\,|\,n)/\mathbb C$ can be identified with the set of $f\in\Lambda(n)$, such that
$f(0)=0$ and $\deg\, f< n$. To define a short grading on $\g=H(n)$ denote by $\g_1$ ($\g_{-1}$) the subspace generated by  the monomials which contain  $\eta_{k+1}$ and do not contain $\eta_1$ 
($\eta_{1}$ and $\eta_{k+1}$, respectively). For $n=2k+1$ the subspaces $\Lambda_1$ and $\Lambda_2$
generated  by all monomials from $\g_{-1}$ which contain or do not contain generator $\eta_0$, respectively,
may be identified with two copies of $\Lambda(2k-2)$ in $\eta_2,\dots,\eta_{k},\eta_{k+2},\eta_{2k}$. 
Moreover $\Lambda_1+\Lambda_2$ is a Jordan superalgebra with respect to multiplication 
$$
x\cdot y=[[a,x],y], \quad a=\eta_0\eta_{k+1}.
$$
Observe that $\cdot$ corresponds to the usual associative product in $\Lambda_1$ and the Poisson bracket
 in $\Lambda_2$. For the case of even $n=2k$ choose a different set of generators $\eta_1$, 
 $\eta'_2=\eta_2-\eta_{n+1}$, $\eta_3$, $\dots$, $\eta_{n+1}$, $\eta'_{n+2}=\eta_2+\eta_{m+1}$, 
 $\eta_{n+3}$, $\dots$, $\eta_{2n}$. The subspace $\Lambda_1$ (the space $\Lambda_2$) is generated by monomials that contain (don't contain) $\eta'_2$. Then $\Lambda_1\oplus\Lambda_2$ is the Kantor double $J_{2n-3}$.

\subsection{Construction of $\spo(0,n)$-modules with short grading.}

As we already mentioned in Introduction representations of Kantor double superalgebra were studied in \cite{Sh2}. 
The authors have shown 
that $Kan(n)$ $n>4$ (over field of characteristic zero) is rigid, i.e. has only regular irreducible supermodule and its opposite. 
The fact that the
$H(n)$, the TKK of $Kan(n)$, has non-trivial central extension $\spo(n)$ was not taken into consideration. In \cite{ZM4} it was corrected, the authors
proved that under the same restriction on characteristic of field and number of variables there exists (up to change of parity) only one-parameter 
family $V(\alpha)$ of irreducible supermodules. Finally in \cite{ShO} it was shown that  every irreducible finite dimensional Jordan 
$Kan(n)$ supermodule for $n\geq 2$ and characteristic of field is different from $2$ is isomorphic  (up to change of parity)  to 
$V(\alpha)$. In this section we study indecomposable $Kan(n)$-modules. 

Assume that $\g=H(n)$, $n>4$ then the universal central extension of $\g$,
$\hat \g$ is isomorphic to the special Poisson algebra: $\spo(0,n)$. It is useful to recall that $\po(0,n)$ is equipped with invariant bilinear form $\omega$ 
$$\omega(f,g)=\frac{\partial}{\partial\xi_1}\dots\frac{\partial}{\partial\xi_n}(fg).$$
The form $\omega$ is symmetric and even (resp. odd) if $n$ is even (resp. odd). It induces the invariant form on $\g=H(n)$.

We also equip $\g$ and $\hat\g$ with a $\Z$-grading
(consistent with $\Z_2$-grading):
\begin{equation}\label{long_grading_poisson}
  \hat{\g}=\hat\g_{-2}\oplus
  {\g}=\g_{-2}\oplus\g_{-1}\oplus\g_{0}\oplus\dots\oplus\g_{(n-3)}.
\end{equation}
where the linear space $\g_{i}$ is generated by monomials of degree $i+2$, $i\geq -2$. Then
$\hat\g_{-2}=\C $ is one-dimensional center, $\g_0$ is orthogonal algebra $\oo(n)$ and  $\g_i$ is 
$\oo(n)$-module $\Lambda^{i+2} V$, $V$ the standard $\oo(n)$-module. This grading is called standard.
We use the notation
$$\g^+:=\bigoplus_{i\geq 0}\g_i,\quad \g^{++}=\bigoplus_{i>0}\g_i.$$

Consider the subalgebra $\pp=\g^+\oplus \hat\g_{-2}\subset\gh$. Let  $N$ be a $\g_0$-module,
extend it to $\pp$-module by setting $\g_i N=0$, $i>0$, $z=t \operatorname {Id}_N$. Then
$Ind^{\hat\g}_{\pp} N = U(\g)\otimes_{U(\pp)} N$ is a $\gh$-module by construction and it is a $\g$-module if $t=0$.
One has the following isomorphism of $\g_0$-modules
\begin{equation}\label{isoshap}
  Ind^{\hat\g}_{\pp} N\simeq N\otimes \Lambda V.
  \end{equation}

  Let $M_t(\lambda)$ be an even simple $\g_0+\g_{-2}$-module with $\oo(n)$-highest weight $\lambda$ and and central charge $t$. 
  We extend it to a
  simple $\pp$-module by setting $\g^{++}M_t(\lambda)=0$.  
  Every simple finite dimensional $\pp$-module is isomorphic to $M_t(\lambda)$ or $M_t(\lambda)^{op}$.
  
Finite dimensional irreducible representations of both $\g$
and $\hat\g$ were described by A. Shapovalov in \cite{Sha1}, \cite{Sha2}.
Let us formulate these results here.
\begin{thm}\label{Shap} Let $n\geq 4$, $\gh=\spo(n)$. 
  \begin{enumerate}
\item    Every simple  $\gh$-module  is a quotient of the induced module
  $Ind^{\hat\g}_{\pp} M_t(\lambda)$ or $Ind^{\hat\g}_{\pp} M_t(\lambda)^{op}$.
   If $t=0$, this quotient is unique,
  we denote it by $L_\lambda$.
  \item Let $\omega_1$ denote the first fundamental weight of $\g_0=\oo(n)$. If the highest weight $\lambda$ is different from $l\omega_1$, $l\in\Z^{\geq 0}$
    then the induced module  $Ind_{\pp}^{\hat\g} M_t(\lambda)$ is simple for every $t$. If $t\neq 0$ then  $Ind_{\pp}^{\hat\g} M_t(0)$ is also simple.
  \item If $k>1$  then  $Ind_{\pp}^{\hat\g} M_0(k\omega_1)$ is an indecomposable module length $4$ with simple socle and cosocle isomorphic to $L_{k\omega_1}$
    and two other simple subquotients isomorphic to $L^{op}_{(k-1)\omega_1}$ and $L^{op}_{(k+1)\omega_1}$.
  \item  There exists a homomorphism
    $\gamma: Ind_{\g^+}^{\g} M_0(2\omega_1)^{op}\to Ind_{\g^+}^{\g} M_0(\omega_1)$ and $\im\gamma$ is an indecomposable module of length $2$ with socle $L_{\omega_1}$ and cosocle $L^{op}_{2\omega_1}$. 
  \item  $Ind_{\pp}^{\hat\g} M_0(0)$ has length $3$ with one dimensional socle and cosocle.
  \item\label{poisson-complex} If $k>0$ and $t\neq 0$ then  $Ind_{\pp}^{\hat\g} M_t(k\omega_1)$ is a direct sum of two non-isomorphic simple modules. There exists an exact complex
    $$Ind_{\pp}^{\hat\g} M_t(0)\to Ind_{\pp}^{\hat\g} M_t(\omega_1)\to Ind_{\pp}^{\hat\g} M_t(2\omega_1)\to\dots $$
    such that the image of every differential is a simple $\gh$-module.
    \end{enumerate}
  \end{thm}

Let  $I_t=Ind_{\pp}^{\hat\g} \mathbb C_t$ be the smallest induced module. Since
$I_t\simeq \Lambda(V)$ as a $\oo$-module, $I_t$ has a short grading.
For $t\neq 0$, the $I_t$ is simple and we denote it by $S(t)$. On the other hand, $I_0$ is the restriction of
the coadjoint module $\po$ to $\spo$ and hence it has length $3$ with one-dimensional trivial module in the cosocle and socle and the coadjoint
$\g$-module at the middle layer of the radical filtration.
If we denote by $S(0)$ the coadjoint module of $\g=H(n)$, then we have the following diagram for the radical filtration of $I_0$ 

$$
\begin{array}{c}
\C\\
\overline{S(0)}\\
\overline{\ \C\ }
\end{array}
\quad {\rm for\ even\ } n \qquad {\rm and}\qquad
\begin{array}{c}
\C\\
\overline{S(0)}\\
\overline{\C^{op}}
\end{array}
\quad {\rm for\ odd\ } n.
$$
Using the form $\omega$ it is easy to check that
$I_0^*\simeq I_0$ for even $n$ and $I_0^*\simeq I_0^{op}$ for odd $n$.

\begin{prop}\label{simplepoisson} Let $n\geq 4$. 
\begin{enumerate}
\item There are no $\spo(n)$ modules which admit very short grading.
\item A simple object in $\spo(n)-\text{mod}_{1}$ is  isomorphic to $\C$, $\C^{op}$,
$S(t)$ or $S^{op}(t)$. 
\end{enumerate}
\end{prop}
\begin{proof} The short $\mathfrak{sl}_2$-subalgebra of $\hat{\g}$ lies in $\g_0=\oo(n)$. Therefore an irreducible quotient of
  $Ind_{\p}^{\hat\g} M_t(\lambda)$
  has a chance to have a short grading only if  $ M_t(\lambda)$ has a short grading as a module over $\g_0$.
  On the other hand, the isomorphism of $\oo$-modules
  $Ind_{\pp}^{\hat\g} M_t(\lambda)\simeq M_t(\lambda)\otimes\Lambda(V)$ implies that the induced module never has a very short grading. Furthermore, 
  for non-zero $\lambda$ the induced module does not have a short grading. On the other hand, the induced module is not irreducible only for $\lambda=k\omega_1$.
  Thus, it remains to consider the cases $\lambda=0$ and $\lambda=\omega_1$. We 
  already considered the former case. Let $\lambda=\omega_1$ and $t\neq 0$. By Theorem~\ref{Shap}(6) 
  $Ind_{\pp}^{\gh} M_t(\omega_1)= S(t)\oplus S'$ for some simple module $S'$ not isomorphic to $S(t)$.
  Since $Ind_{\pp}^{\gh} M_t(\omega_1)$ does not have the short grading, the same is true for $S'$. For $t=0$  $S(0)$ is isomorphic to $L^{op}_{\omega_1}$ and the statement follows from Theorem~\ref{Shap}(1).
\end{proof}

\begin{rem}
  It follows from Proposition \ref{simplepoisson}(1) that category $Kan(n)$-mod$_{\frac12}$ is trivial. This is a consequence of the fact that
 $Kan(n)$ for
$n\geq 2$ is exceptional, \cite{KMC}.
\end{rem}

\begin{rem}\label{remchar} Note that $S(t)$ is isomorphic to $\Lambda V=\oplus_{i=0}^{n}\Lambda^i V$ as a $\g_0$-module and $S(0)$ is isomorphic to $\oplus_{i=1}^{n-1}\Lambda^i V$.
  \end{rem}

\subsection{The case of non-zero central charge}
\begin{lem}\label{extnonzero} If $t\neq 0$ then
  $$\Ext^1(S(t),S^{op}(t))=0,\quad\Ext^1(S(t),S(t))=\C. $$
  \end{lem}
  \begin{proof} Note that for even $n$ the first assertion follows  {from Lemma~\ref{weighargument}}.
    Let us prove the first assertion for odd $n$. By \eqref{shapiro's_lemma} we have
    $$\Ext^1(S(t),S^{op}(t))=\Ext_{\pp}^1(\C_t, S^{op}(t))=\Ext_{\g^+}^1(\C, S^{op}(t)).$$
    The latter equality follows from the fact that the center always acts semisimply on an extension of two non-isomorphic simple modules.

    Every finite-dimensional $\g_0$-module is semisimple. Therefore we have to show that the relative Lie algebra cohomology
    $H^1(\g^+,\g_0; S^{op}(t))$ vanishes. Let us write the cochain complex calculating this cohomology:
$$0\rightarrow C^0=\Hom_{\g_0}(\C,S^{op}(t))\xrightarrow{d_1}C^1=\Hom_{\g_0}(\g^{++},S^{op}(t))\xrightarrow{d_2}C^2=\Hom_{\g_0}(\Lambda^2\g^{++},S^{op}(t))\xrightarrow{d_3}\dots$$
    By Remark \ref{remchar} $\dim C^0=1$. By Theorem \ref{Shap} $H^0(\g^+,\g_0;S^{op}(t))=\C^{op}$. Therefore $d_1\neq 0$. To determine the kernel of $d_2$ we observe that
    $\g_1$ generates $\g^{++}$, hence any $1$-cocycle is determined by its value on $\g_1$. Thus,
    $\Ker d_2$ is a subspace in $\Hom_{\g_0}(\g_1,S(t)^{op})$ and the latter space is one-dimensional. Hence
    $\im d_1=\Ker d_2$ and the assertion is proved.

    Now we will deal with the second assertion. We observe that $S(t)$ has a non-trivial self-extension  given by the induced module
    $Ind^{\gh}_\pp\C[z]/(z-t)^2.$ Therefore it suffices to prove that there is no self-extensions of $S(t)$ on which $z$ acts semisimply.
    Then again by Shapiro's lemma it suffices to prove $H^1(\g^+,\g_0; S(t))=0$.

    Consider again the chain complex:
    $$0\rightarrow C^0=\Hom_{\g_0}(\C,S(t))\xrightarrow{d_1}C^1=\Hom_{\g_0}(\g^{++},S(t))\xrightarrow{d_2}C^2=\Hom_{\g_0}(\Lambda^2\g^{++},S(t))\xrightarrow{d_3}\dots.$$
    If $n$ is odd then $\dim C^0=1$ and $H^0(\g^+,\g_0,S(t))=\C$, hence $d_1=0$. By the same argument as above a $1$-cocycle is determined by
    its value on $\g_1$. By Remark \ref{remchar} 
    $\dim\Hom_{\g_0}(\g_1,S(t))=1$, which gives  $\dim\Ker d_2\leq 1$, in other words, there is exactly one up to proportionality
    $\varphi\in  \Hom_{\g_0}(\g_1,S(t))$. In the monomial basis of $\gh$ the map $\varphi$ can be written in the following form: fix  $v\in \C_t$ then
    $$\varphi(\xi_i\xi_j\xi_k)=\xi_i(\xi_j(\xi_k v)).$$ We claim that $\varphi$ can not be extended to a one cocylce in $C^1$. Indeed, let
    $u=\xi_1\xi_2\xi_3$, then $\{u,u\}=0$ and the cocycle condition on
    $\varphi$ implies $u\varphi(u)=0$. But the direct computation shows
    $$u(\xi_1(\xi_2(\xi_3v)))=\{u,\xi_1\}(\xi_2(\xi_3v))-\xi_1(\{u,\xi_2\}(\xi_3 v)+\xi_1\xi_2(\{u,\xi_3\}v)).$$
    Since $\{u,\xi_3\}\subset\g_0 v=0$, the last summand is zero. Continue the computation and get
    $$u(\xi_1(\xi_2(\xi_3v)))=(\xi_2\xi_3)(\xi_2(\xi_3v))-\xi_1((\xi_1\xi_3)(\xi_3 v))=\xi^2_3v-\xi^2_2v+\xi_1^2v=tv\neq 0.$$
    That proves $\Ker d_2=0$.

    If $n$ is even the proof goes similarly to the case of an odd $n$. In this case we have $H^0(\g^+,\g_0,S(t))=\C$, $\dim C^0=2$ and hence
    $\im d_1$ is one-dimensional. Furthermore $\dim\Hom_{\g_0}(\g_1,S(t))=2$. We can choose a basis $\varphi,\psi$ in $\Hom_{\g_0}(\g_1,S(t))$ such that $\varphi$ is given by the same formula as in the
    odd case and $\psi\in d_1(C^0)$.
    The same calculation shows $\varphi$ does not extend to a cocycle. This completes the proof.
  \end{proof}

\begin{prop}\label{nonzeroblocks}
  If $t\neq 0$ the category $\gt$ has two equivalent blocks $\Omega^+_t$ and $\Omega^-_t$. The equivalency of these blocks
  is established by the change parity functor. Both  $\Omega^+_t$ and $\Omega^-_t$ contain only one up to isomorphism simple object  $S(t)$ and $S(t)^{op}$ respectively. Moreover,
  $\Omega^+_t$ is equivalent to the category $\C[x]$-modules with nilpotent action of $x$.
\end{prop}
\begin{proof} The first two assertions follow immediately from Proposition \ref{simplepoisson} and Lemma \ref{extnonzero}. To prove the last assertion
  we consider the subcategory
  $F^n(\gt)$ of modules annihilated by $(z-t)^n$. Then $Ind^{\gh}_{\pp}\C[z]/(z-t)^n$ is projective in  $F^n(\gt)$ by Lemma \ref{extnonzero}
  and  its indecomposability. Since every object of $\gt$ lies in some $F^n(\gt)$ the statement follows.
  \end{proof}
  \begin{Cor}\label{indecpononzero} If $t\neq 0$ every indecomposable module in $\gt$ is isomorphic to $Ind^{\gh}_{\pp}\C[z]/(z-t)^n$ or
$(Ind^{\gh}_{\pp}\C[z]/(z-t)^n)^{op}$.
  \end{Cor} 
  \begin{Cor}\label{jordannonzero}  If $t\neq 0$, then every block in the category $\JJ^{\,t}$ is equivalent to  the category of $\C[x]$-modules with  nilpotent action of $x$.
\end{Cor}

\subsection{The case of zero central charge}

\begin{lem}\label{poext} 
\begin{enumerate}
    \item If $n$ is even then $\Ext^1(\C,S(0))=\Ext^1(S(0),\C)=\C^2$ and $\Ext^1(\C^{op},S(0))=\Ext^1(S(0),\C^{op})=0$.
      \item If $n$ is odd then $\Ext^1(\C,S(0))=\Ext^1(S(0),\C)=\Ext^1(\C^{op},S(0))=\Ext^1(S(0),\C^{op})=\C$.
      \end{enumerate}
    \end{lem}
    \begin{proof} It suffices to show that $\Ext^1(\C,S(0))=\C^2$ for even $n$ and $\Ext^1(\C,S(0))=\C=\Ext^1(\C^{op},S(0))$ since the rest follows from
      duality and Lemma~\ref{weighargument}. Both statement follow from the well-known fact about derivation superalgebra. Indeed, it is shown in \cite{Kac2} that
      $\Der\g/\g=\C^2$ for even $n$ and $\Der\g/\g=\C^{1|1}$ for odd $n$. These derivations are given by {the} Poisson bracket with $\xi_1\dots\xi_n$ and
      by the commutator with the Euler vector field $\sum_{i=1}^n\xi_i\partial_i$. The latter derivation defines the standard grading of $\g$ and $\gh$.
    \end{proof}
To compute other extensions between simple modules we first consider only extensions in $\g$-mod$_1$ which we denote  $\Ext^1_\g$.
    \begin{lem}\label{crucial} Let $M=Ind_{\g^+}^{\g} M_0(\omega_1)$
      and $n>5$. Then $\Ext^1_\g(M,S(0))=\Ext^1_\g(M,S(0)^{op})=0$. In the case of $n=5$ we have $\Ext^1_\g(M,S(0)^{op})=0$ and $\Ext^1_\g(M,S(0))=\C$. 
    \end{lem}
    \begin{proof} Let us start with the case of even $n$. The weight argument, Lemma~\ref{weighargument}, implies 
    $\Ext^1_\g(M,S(0)^{op})=0$. Let us show that
      $\Ext^1_\g(M,S(0))=0$. By Shapiro's lemma
      $$\Ext^1_\g(M,S(0))=\Ext^1_{\g^+}(M_0(\omega_1),S(0))=H^1(\g^+,M_0(\omega_1)^*\otimes S(0))=
      H^1(\g^+,\g_0; M_0(\omega_1)^*\otimes S(0)).$$
      The computations are similar to ones in the proof of Lemma \ref{extnonzero}. We are looking for 
      $\varphi\in \Hom_{\g_0}(\g_1\otimes M_0(\omega_1), S(0))$
      which can be extended to a cocycle in $\Hom_{\g_0}(\g^{++}\otimes M_0(\omega_1), S(0))$.
      We use the fact that $M_0(\omega_1)=V$ is the standard representation of $\g_0=\oo(n)$ and
      $$S(0)=\bigoplus_{i=1}^{n-1}\Lambda^i(V).$$
      Therefore it is not hard to compute that  $\Hom_{\g_0}(\g_1\otimes M_0(\omega_1), S(0))$ is a $4$-dimensional and we can write down a basis
      $\{\varphi_j\,|\,j\leq 4\}$ homogeneous with respect
      to the standard grading. We identify $V$ with $\Lambda^1(V)\subset S(0)$ and denote by $\ \bar{\ \ }:V\to\Lambda^{n-1}(V)\subset S(0)$
      the natural $\g_0$-isomorphism. We set for every $f\in\g_1,x\in V$
      $$\varphi_1(f,x)=L_f(x),\quad \varphi_2(f,x)=fx,\quad \varphi_3(f,x)=L_f^{(2)}(\bar x), \quad \varphi_4(f,x)=L_f^{(3)}(\bar x),$$
      where
      $$L_f=\sum_{i=1}^n\partial_i(f)\partial_i,\quad L_f^{(2)}=\sum_{i<j}(\partial_i\partial_j(f))\partial_j\partial_i,
      \quad L_f^{(3)}=\sum_{i<j<k}(\partial_i\partial_j\partial_k(f))\partial_k\partial_j\partial_i.$$
      We first notice that $\varphi_1$ is a coboundary by
      construction, thus we can assume without loss of generality that the restriction of our cocycle on $\g_1$ is given by
      $\varphi=c_2\varphi_2+c_3\varphi_3+c_4\varphi_4$. Let us show that if $\varphi$ extends to a cocycle then $c_1=c_2=c_3$.

      First, we take $f=\xi_1\xi_2\xi_3$, $x=\xi_1$, then $\{f,f\}=0$. Hence $\varphi(\{f,f\},x)=2\{f,\varphi(f,x)\}=0$. But
      $\varphi_2(f,x)=\varphi_4(f,x)=0$ and
      $$2\{f,\varphi(f,x)\}=2c_3\{f,\varphi_3(f,x)\}=2c_3\{\xi_1\xi_2\xi_3,\xi_1\xi_4\xi_5\dots\xi_n\}=2c_3\xi_2\xi_3\xi_4\xi_5\dots\xi_n.$$
      This implies $c_3=0$.
      Next we take $x=\xi_1$, $f=\xi_1\xi_5\xi_6+\xi_2\xi_3\xi_4$. Again we must have $2\{f,\varphi(f,x)\}=0$. Therefore
$$\{f,\varphi(f,x)\}=-c_2\{\xi_1\xi_5\xi_6+\xi_2\xi_3\xi_4,\xi_1\xi_2\xi_3\xi_4\}+c_4\{\xi_1\xi_5\xi_6+\xi_2\xi_3\xi_4,\xi_5\xi_6\dots\xi_n\}=
      -c_2\xi_5\xi_6\xi_2\xi_3\xi_4=0.$$
      Thus $c_2=0$.
     
      It remains to check that $\varphi_4$ can not be extended to a cocycle. Let $f=\xi_1(\xi_2\xi_3+\xi_4\xi_5)$, $u=\{f,f\}=2\xi_2\xi_3\xi_4\xi_5$, $x=\xi_2$.
      Then
      $$\varphi_4(f,x)=\xi_3\alpha, \quad\alpha=\xi_6\dots\xi_n,$$
      $$\varphi_4(u,x)=2\{f,\varphi_4(f,x)\}=2\{f,\xi_3\alpha\}=2\xi_1\xi_2\alpha.$$
      Let $g=\xi_2(\xi_1\xi_3+\xi_4\xi_5)$, $v=\{g,g\}=2\xi_1\xi_3\xi_4\xi_5$. Then $\varphi_4(g,x)=0$, hence $\varphi_4(v,x)=0$. On the other hand, $\{u,v\}=0$,
      therefore
    $$0=\varphi_4(\{u,v\},x)=\{u,\varphi_4(v,x)\}-\{v,\varphi_4(u,x)\}=-\{2\xi_1\xi_3\xi_4\xi_5,2\xi_1\xi_2\alpha\}=4\xi_3\xi_4\xi_5\xi_2\alpha.$$
    A contradiction.

      The case of odd $n$ for $n\geq 7$ can be proven similarly. The only difference is that both $\Hom_{\g_0}(M_0(\omega_1), S(0))$ and
      $\Hom_{\g_0}(M_0(\omega_1), S(0)^{op})$ are $2$-dimensional, the former space is spanned by $\varphi_3,\varphi_4$ and the latter is
      spanned by $\varphi_1,\varphi_2$.

      Finally, for $n=5$ all above arguments are applicable except the proof that $c_2=0$. 
      In this case if we set $\varphi_2(\g_2,M_0(\omega_1))=0$ we obtain a cocycle which
      gives a non-trivial extension in  $\Ext^1_\g(M,S(0)^{op})$.  
    \end{proof}
    It follows from \cite{Sha1} Theorem 3 that there exists a homomorphism
    $\gamma: Ind_{\g^+}^{\g} M_0(2\omega_1)^{op}\to Ind_{\g^+}^{\g} M_0(\omega_1)$ and $\im\gamma$ is an indecomposable module of length $2$ with socle $L_{\omega_1}$ and
    cosocle $L^{op}_{2\omega_1}$.  Let $Q$ denote the quotient of $M=Ind_{\g^+}^{\g} M_0(\omega_1)$ by $\im\gamma$. 
    \begin{lem}\label{projaux} Let $n>5$. We have $\Ext^1_\g(Q,S(0))=\Ext^1_\g(Q,S(0)^{op})=0$. 
    \end{lem}
    \begin{proof} Consider the exact sequence
      $$0\to\im\gamma\to M\to Q\to 0.$$
      Let $S=S(0)$ or $S(0)^{op}$. Consider the corresponding long exact sequence
      \begin{equation}\label{exactseq}
        \dots\to \Hom_\g(\im\gamma,S)\to \Ext^1_\g(Q,S)\to \Ext^1_\g(M,S)\to \dots
        \end{equation}
      We have  $\Hom_\g(\im\gamma,S)=0$ and  $ \Ext^1_\g(M,S)=0$ if $n>5$ or $S=S(0)$. Therefore  $\Ext^1_\g(Q,S)=0$.
      \end{proof}
  \begin{prop} Let $t=0$ and $n>5$. Then $Q$ is projective in the category $\ggm$.
\end{prop}
\begin{proof} It suffices to check that $\Ext^1_\g(Q,S)=0$ for all simple $S$ in $\ggm$. For $S=S(0)$ or $S^{op}(0)$ this is Lemma \ref{projaux}. For $S=\mathbb C$
  consider the exact sequence $0\to R\to Q\to F\to 0$ where $F=S(0)^{op}$  and  $R=\mathbb C^2$ for even $n$ , $R=\mathbb C\oplus\mathbb C^{op}$
  for odd $n$. The corresponding long exact sequence degenerates
  $$0\to\Hom_\g(R,\mathbb C)\xrightarrow{\theta}\Ext^1_\g(F,\mathbb C)\to \Ext^1_\g(Q,\mathbb C)\to\Ext^1_\g(R,\mathbb C)=0.$$
  By Lemma \ref{poext} $\theta$ is an isomorphism and hence $\Ext^1_\g(Q,\mathbb C)=0$. The case $S=\mathbb C^{op}$ is similar.
\end{proof}
Let $I^{(m)}:=Ind^{\gh}_{\pp}\mathbb C[z]/(z^{m+1})$ and $J^{(m)}$ be the unique maximal submodule of $I^{(m)}$ and $Q^{(m-1)}$ be the quotient of $J^{(m)}$ by
the unique maximal submodule in $Ind^{\gh}_{\g^+}z^m\subset I^{(m)}$.  
\begin{lem}\label{qext} Let $n>5$, $m\geq 1$. Then $z^iQ^{(m-1)}/z^{i+1}Q^{(m-1)}$ is isomorphic to $Q$ for $i=0,\dots,m$. Moreover, $Q^{(m-1)}$ is projective in 
$F^1(\gm^{\,0})$.
\end{lem}
\begin{proof} The first assertion is a consequence of the isomorphism $z^jQ^{(m-1)}/z^{j+1}Q^{(m-1)}\simeq z^iQ^{(m-1)}/z^{i+1}Q^{(m-1)}$ and the observation that
  $Q^{(m-1)}/zQ^{(m-1)}$ is indecomposable of length $3$ with the cosocle $S(0)^{op}$ and socle $\mathbb C^2$ (resp. $\mathbb C\oplus\mathbb C^{op}$) for even (resp., odd)
  $n$. Lemma \ref{poext} implies that the module with these properties is unique up to isomorphism, hence it is isomorphic to $Q$.

  The second assertion follows from Lemma \ref{projaux} by induction on $m$.
\end{proof}

Now we are going to prove the following
  \begin{thm}\label{jordanzero}  Let $n\geq5$. The category $\JJ^{\,0}$ has two blocks, each of these blocks is equivalent to  the category of $\C[x]$-modules with
    nilpotent action of $x$.
  \end{thm}

  \begin{proof} For $n\geq 6$ it follows from the fact that $Jor(Q^{(m-1)})$ is projective in the corresponding subcategory $\JJ$. Now we consider  the case $n=5$. We would like to show that
    the module $Q$ is a projective cover  of $S(0)^{op}$  in $\ggm^{\,0}$. It suffices to show that $\Ext^1_\g(Q,S(0))=0$.

    Consider a unique up to proportionality
$$\varphi\in\Hom_{\g_{0}}(\g_1\otimes M_0(\omega_1),M_0(\omega_1)^{op}).$$ This map defines $\g^+$ module structure on
$\bar{M}_0(\omega_1):=M_0(\omega_1)\oplus M_0(\omega_1)^{op}$, assuming that $\g_2$ acts by zero.  Note that the extension of $Ind^\g_{\g^+}{M}_0(\omega_1)$ by $S(0)$ is a quotient
of $Ind^\g_{\g^+}\bar{M}_0(\omega_1)$ by the maximal proper submodule of   $Ind^\g_{\g^+}{M}_0(\omega_1)^{op}$. Therefore the exact sequence (\ref{exactseq}) implies that a non-trivial extension of
 $Q$ by $S(0)$ is a quotient of $Ind^\g_{\g^+}\bar{M}_0(\omega_1)$. We will show that every quotient of $Ind^\g_{\g^+}\bar{M}_0(\omega_1)$ which lies in $\ggm^{\,0}$ is in fact a quotient of $Ind^\g_{\g^+}M_0(\omega_1)$.
Indeed, consider a quotient $Ind^\g_{\g^+}\bar{M}_0(\omega_1)/N$ for some submodule $N$.
 Let $v$ and $v'$ be $\g_0$ highest weight vectors in $M_0(\omega_1)$ and  ${M}_0(\omega_1)^{op}$ respectively and $x\in\g_{-1}$ be a $\g_0$-highest
  vector. Then $N$ contains $xv$ and $xv'$ as the weight of these vectors is $2\omega_1$. Let $y\in\g_2$ be the lowest weight vector.
  Then
  $$yxv=xyv+[x,y]v=[x,y]v=v'.$$ 
  Therefore the whole  $Ind^\g_{\g^+}{M}_0(\omega_1)^{op}$ is contained in $N$. Now one can complete the proof of the theorem as in the case $n\geq 6$.
\end{proof}

    \begin{Cor}\label{jordanzeroindec}  Let $n\geq 5$. Every indecomposable module  in the category $\JJ^{\,0}$ is isomorphic to
$Jor(Q^{(m-1)})$ or $Jor(Q^{(m-1)})^{op}$.   
\end{Cor}

\section{Representations of $M^+_{1,1}$.}

Let $M_{n,m}$ be the associative superalgebra
$$
M_{n,m}=\left\{\left[\begin{array}{cc} A&B\\ C&D\end{array}\right]\,|\, A\in M_n,\ D\in M_m,\ B\in M_{n\times m}, C\in M_{m\times n}\right\}=
\left[\begin{array}{cc} A&0\\ 0&D\end{array}\right]_{\bar{0}}\oplus \left[\begin{array}{cc} 0&B\\ C&0\end{array}\right]_{\bar{1}}.
$$
Jordan (resp. Lie) superalgebra $M^+_{n,m}$ (resp. $\gl(m,n)$) has the same underlying vector superspace and multiplication is a symmetric (resp. Lie) product $A\cdot B=\frac12(AB+BA)$ (resp. $[A,B]=AB-BA$).
These superalgebras are also related to each other via the TKK 
construction.

Denote by $E_{ij}$ $1\leq i,j\leq 4$, the standard basis of $\gl(2|2)$ consisting of the elementary matrices.
We have the direct sum decomposition
$$
\gl(2|2)=\ssl(2|2)\oplus\C(E_{11}+E_{22}-E_{33}-E_{44}),
$$ 
where $\ssl(2|2)$ is the subalgebra of $\gl(2|2)$ of matrices with zero supertrace. 

Next, the element $z_0=\frac12(E_{11}+E_{22}+E_{33}+E_{44})$ is central in $\ssl(2|2)$ and the quotient of $\ssl(2|2)$
by the ideal generated by $z_0$ is the simple Lie superalgebra $\g=\psl(2|2)$.  Then $Lie(M^+_{1,1})=\psl(2|2)$, see \cite{Kac1}. 
The short (Jordan) $\ssl(2)$-grading is given by $h=E_{11}-E_{22}+E_{33}-E_{44}$ and $\ssl(2)$ subalgebra is spanned by $h$, $E_{12}+E_{34}$
and $E_{21}+E_{43}$.

We fix the standard basis of the Cartan subalgebra of $\g$:
$$
h_1=E_{11}-E_{22}, \quad h_2=E_{33}-E_{44}. 
$$
Note that  $\g$ has an invariant symmetric  form $(\,,\, )$ induced by the form $\operatorname{str}XY$ on $\mathfrak{gl}(2|2)$. Therefore $H^2(\g,\mathbb C)$ and $H^1(\g,\g)=\operatorname{Der}(\g)/\g$ are isomorphic. Furthermore, \cite{Kac2},
$\operatorname{Der}(\g)/\g$ is isomorphic to $\ssl(2)$, and the action of $\ssl(2)$ on $H^2(\g,\mathbb C)$ equips the latter with the structure of the
adjoint representation. Therefore the universal central extension $\hat\g$ has a 3-dimensional center $Z$ with the basis $z_{-1},z_0,z_1$  such that
\begin{equation}
[E_{13},E_{24}]=-[E_{23},E_{14}]=z_1, \qquad [E_{31},E_{42}]=-[E_{32},E_{41}]=z_{-1}.
\end{equation}
Furthermore, the Lie algebra $\ssl(2)$ acts on $\hat\g$ by derivations, \cite{Vera2}. If $e,h,f$ is the standard $\ssl(2)$-triple, then
$$H(z_i)=2iz_i,\quad E(z_i)=z_{i+1},\quad F(z_i)=z_{i-1},$$ 
$$E\left[\begin{array}{cc}A&B\\C&D\end{array}\right]=\left[\begin{array}{cc}0&B+C^*\\0&0\end{array}\right],\ 
H\left[\begin{array}{cc}A&B\\C&D\end{array}\right]=\left[\begin{array}{cc}0&B\\-C&0\end{array}\right],\ 
F\left[\begin{array}{cc}A&B\\C&D\end{array}\right]=\left[\begin{array}{cc}0&0\\C+B^*&0\end{array}\right],$$
where $A,B,C,D$ are $2\times 2$-matrices and
$\left[ \begin{array}{cc}a&b\\c&d\end{array}\right]^*=\left[\begin{array}{cc}d&-b\\-c&a\end{array}\right]$. 

The eigenspace decomposition of $ad\,H$ defines a short grading on $\hat\g$
consistent with the superalgebra grading
$$
\hat\g=\hat\g_{-2}\oplus \hat\g_{-1}\oplus\hat\g_{0}\oplus\hat\g_{1}\oplus\hat\g_{2},
$$
where 
$$
\hat\g_{-1}=\left[\begin{array}{cc}0&0\\C&0\end{array}\right], \quad \hat\g_{0}=\left[\begin{array}{cc}A&0\\0&D\end{array}\right]\oplus\C z_0,\quad  \hat\g_{1}=\left[\begin{array}{cc}0&B\\0&0\end{array}\right] \quad  \text{and}\quad \hat\g_{\pm2}=\C z_{\pm}.
$$

This action can be lifted the action of the group $SL(2)$ as follows.
For any $\phi=\left[\begin{array}{cc} u&v\\ w&z \end{array}\right]\in SL(2)$ each element in $\g_{\bar 0}$ is stable under
$\phi$ while the action on $\g_{\bar 1}$ is determined by
\begin{equation}\label{autom-pairs}
\phi(E_{14})=uE_{14}+vE_{32}, \qquad \phi(E_{32})=wE_{14}+zE_{32}.
\end{equation}

Let $M$ be a finite-dimensional irreducible representation of $\hat\g$ then by twisting the action of $\hat\g$ on $M$ 
by $\phi$ we obtain another irreducible representation $M^\phi$ of $\hat\g$. Moreover, since $M$ is irreducible, it admits central character $\chi$,
i.e., every central central element $z$ acts on $M$ as the scalar $\chi(z)$. If
$\chi(z_0)=c$, $\chi(z_{-1})=p$ and $\chi(z_1)=k$, then  $M^\phi$ admits central character $\phi(\chi)$ defined by new coordinate components
$c'$ $p'$ and $k'$ 
$$
\left[\begin{array}{cc} c'&-k'\\ p'&-c' \end{array}\right]=\left[\begin{array}{cc} u&v\\ w&z \end{array}\right]
\left[\begin{array}{cc} c&-k\\ p&-c \end{array}\right]\left[\begin{array}{cc} u&v\\ w&z \end{array}\right]^{-1}.
$$

\subsection{Simple modules in $\gm$ and $\gmhh$}
Irreducible modules for $M^+_{1,1}$ were studied in \cite{ZM} and recently in \cite{MS}. The classification is obtained for any field of characteristic $\neq 2$. In this section we describe categories $M^+_{1,1}\text{-mod}_{\frac 12}$
and $M^+_{1,1}\text{-mod}_1$ via corresponding categories $\gm$ and $\gmhh$ over the field $\C$.

The category $\gmnada$ of all finite dimensional representations decomposes into blocks
$\gmnada^{\chi}$ and $(\gmnada^{\chi})^{op}$
according to the generalized central character. The action of $SL(2)$ allows to define the canonical equivalence of blocks 
$\gmnada^{\,\chi}$ and $\gmnada^{\,\phi(\chi)}$. Form the description of $SL(2)$-orbits in the adjoint 
representation it is clear
that we can reduce the study of blocks to the three essential cases
\begin{enumerate}
\item Semisimple: $k=p=0,\,\, c\neq 0$;
\item Nilpotent: $c=k=0,\,\, p\neq 0$;
  \item Trivial central character $k=p=c=0$,
  \end{enumerate}
The Lie superalgebra $\hat\g/\operatorname{Ker}\chi$ is isomorphic to $\ssl(2|2)$, $\spo(0,4)$ and $\psl(2|2)$ respectively.  

The following Lemma is straightforward but very important.
\begin{lem}\label{equivmat} The group $SL(2)$ acts on the isomorphism classes of modules in $\gm$ and in 
$\gmhh$ by twist $M\mapsto M^g$, $g\in SL(2)$.
  Moreover, if $M\in\gm^{\,\chi}$ (resp., $\gmhh^{\,\chi}$) then $M^g\in\gm^{\,g(\chi)}$ (resp., $\gmhh^{\,g(\chi)}$). In particular, the categories  $\gm^{\,\chi}$ and $\gmhh^{\,\chi}$
are equivalent to the categories  $\gm^{\,g(\chi)}$ and $\gmhh^{\,g(\chi)}$ respectively.
\end{lem}
Now we are going to classify simple objects of $\gm^{\,\chi}$ and $\gmhh^{\,\chi}$. Denote by $O_1$ (resp. $O_2$)
the $SL(2)$-orbit  defined by the equation  $c^2-kp=1$ (resp. $c^2-kp=4$).

\begin{thm}\label{veryshortsl} $\gmhh^{\,\chi}$  is nonempty if and only if $\chi$ is semisimple and lies on  $O_1$.
 If $c=1,\,\,k=p=0$, then $\gmhh^{\,\chi}$ has two up to isomorphism simple object $V$ and $V^{op}$, where $V$ is the standard $\ssl(2|2)$-module. For any
    $\chi\in O_1$, the subcategory $\gmhh^{\,\chi}$ has two up to isomorphism simple objects $V^g$ and $(V^{op})^g$ for a suitable automorphism $g\in SL(2)$. 
\end{thm}
  \begin{proof} In the nilpotent and trivial case we can use the results of Shapovalov and the previous Section to see that $\po(0,4)$ and
  $H(4)\simeq \psl(2|2)$ do not have modules with very short grading.

    Assume now that $\chi$ is semisimple and furthermore $k=p=0$. We can make these assumptions without loss of generality due to Lemma \ref{equivmat}.
    Thus, our problem is reduced to the classification of simple $\ssl(2|2)$-modules  with very short grading. Let $L$ be such a module.
    Consider a Borel subalgebra $\g_0\oplus\g_1$ of $\ssl(2|2)$ with two even simple roots
    $\beta_1,\beta_2$ and one odd simple root $\alpha$. We may choose the simple coroots $\beta_1^{\vee}$ and $\beta_2^{\vee}$ so that $h=\beta^\vee_1+\beta_2^\vee$.
    Let $\lambda$ be a highest weight of $L$ with respect to this Borel subalgebra.  Observe that 
    \begin{equation}\label{formulaec}
    c=(\lambda, 2\alpha+\beta_1-\beta_2)
    \end{equation}The condition of $L$ to have a very short grading implies
    $\lambda(h)=1$, hence we have two possibilities
\begin{enumerate}
\item    $\lambda(\beta_1^\vee)=1$, $\lambda(\beta_2^\vee)=0$;
\item $\lambda(\beta_1^\vee)=0$, $\lambda(\beta_2^\vee)=1$.
\end{enumerate}
Note that we also have $\alpha(h)=-2$. Thus, if $v$ is highest weight vector and $X\in\g_{-\alpha}$ is a root vector. We must have $Xv=0$. Therefore
$(\lambda,\alpha)=0$. Hence in the first case $L$ isomorphic to the standard representation of $\ssl(2|2)$ and in the second case $L$ is isomorphic to
the dual of the standard representation with switched parity. The action by the element $\left[\begin{array}{cc} 0&1\\-1&0 \end{array}\right]\in SL(2)$ maps one representation to another. Hence the statement of the Lemma.
  \end{proof}
 
 \begin{Cor} $\Jhalf^{\,\chi}$ is nonempty if and only if $\chi$ is semisimple and lies on  $O_1$. Let $\chi=(c,p,k)\in O_1$, $c\neq 0$ then
 there are two up to isomorphism simple object $W$ and $W^{op}$ in $\Jhalf^{\,\chi}$ where $W=\langle w_1,w_2\rangle$ is $(1,1)$-dimensional space and the action of $M_{1,1}^+$ is given$$
\begin{array}{llll}
 E_{ii}w_j=\delta_{i,j} w_j& i,j=1,2&&\\
 E_{12} w_1= (c-1) w_2&  E_{21} w_1= pw_2 &
 E_{12} w_2= kw_1 &  E_{21} w_2=(c-1)w_1
 \end{array}
 $$
 \end{Cor}
 \begin{proof} Let $c=1$, $p=0=k$. Consider standard $\ssl(2|2)$ module $V$ then $Jor(V)=W$, where $W$ is standard module for
 $M_{1,1}^+$. Suppose that $\chi'=(c',p',k')\in O_1$ then the element of $SL(2)$ which takes $\chi$ to $\chi'$ is $\left[\begin{array}{cc}
 k'& c'-1\\ c'-1&p'\end{array}\right]$. The rest follows from applying this automorphism to $W$.
  \end{proof}

  Now let us assume that $k=0$. Let $\pp=\g_0\oplus\g_1\oplus \mathbb Cz_0\oplus \mathbb Cz_{-1}$. We denote by $K_\chi$ the induced module
  $Ind^\g_{\pp}\mathbb C_{\chi}$. Note that $K_{\chi}$ is an object in $\gm^{\,\chi}$.
  
  \begin{thm}\label{shortsl}
    (a) If $\chi\neq 0$ and $\chi{\notin O_2}$, then
    $\gm^{\,\chi}$  has two up to isomorphism simple
    modules. In the case $k=0$ these modules are isomorphic to   $K_{\chi}$ and $K_{\chi}^{op}$. If $k\neq 0$, the simple objects of $\gm^{\,\chi}$
    are obtained by a suitable twist.
    
    (b) If  $\chi=0$, then $\gm^{\,\chi}$ has four up to isomorphism simple modules: $ad, ad^{op}, \C,\C^{op}$.

    (c) If $c=2,k=p=0$, then
    $\gm^{\,\chi}$ has four up to isomorphism simple modules $S^2V$, $\Lambda^2 V$, $(S^2V)^{op}$ and $(\Lambda^2V)^{op}$.  For an arbitrary $\chi\in O_2$ simple
    objects of $\gm^{\,\chi}$ are obtained from those four by a suitable twist.
  \end{thm}
  \begin{proof} If $\chi$ is nilpotent or trivial the result is indeed a consequence of Proposition~\ref{simplepoisson}.

    Now we will deal with semisimple case and assume that $k=p=0$.  We use notation of the proof of Theorem \ref{veryshortsl}. Assume that $L$ is simple $\g=\ssl(2|2)$-module with
    short grading. Then as in the proof of the theorem we can easily conclude there are at most four possibilities for the highest weight $\lambda$ of $L$:
\begin{enumerate}
  \item    $\lambda(\beta_1^\vee)=2$, $\lambda(\beta_2^\vee)=0$;
    \item $\lambda(\beta_1^\vee)=0$, $\lambda(\beta_2^\vee)=2$;
    \item $\lambda(\beta_1^\vee)=\lambda(\beta_2^\vee)=1$;
    \item $\lambda(\beta_1^\vee)=\lambda(\beta_2^\vee)=0$.
    \end{enumerate}
    By the same argument as in the proof  of Theorem \ref{veryshortsl} we obtain the condition $(\lambda,\alpha)=0$ in the first three cases.
    This gives $L\simeq S^2V$, $L\simeq \Lambda^2 V^*$ and $L\simeq ad^{op}$ in the cases (1), (2) and (3) respectively. In case $(4)$
    $L$ is the unique quotient of the Kac module $K_{\chi}$. Recall that the latter module is simple if and only
    if $\lambda$ is typical, i.e.,
    $$(\lambda,\alpha)\neq 0,\ (\lambda,\alpha+\beta_1)+1\neq 0,\ (\lambda,\alpha+\beta_2)-1\neq 0,\ (\lambda,\alpha+\beta_1+\beta_2)\neq 0.$$
   For atypical case we have the following three possibilities
    \begin{enumerate}
    \item $(\lambda,\alpha)=1$, then $L$ is isomorphic to $\Lambda^2V$;
    \item $(\lambda,\alpha)=-1$, then $L$ is isomorphic to $S^2V^*$;
    \item $(\lambda,\alpha)=0$, then $L$ is the trivial module $\mathbb C$.
            \end{enumerate}
The first two cases will give $c=\pm 2$. The twist by $SL(2)$ completes the proof.
\end{proof}

Next we will calculate $Jor(K_{\chi})$.
Let $\chi$, $\pp$ and $\C_{\chi}$ as above. Then $\C_{\chi}=\C v$ where $h_1v=h_2v=E_{12}v=E_{34}v=z_1 v=0$, while $z_0 v=c$ and $z_{-1}v=p$.
Then  the basis of 
$K_{\chi}\simeq Ind^\g_{\pp}\mathbb C_{\chi}$ is formed by
the vectors
$$
E_{41}^{\theta_1}E_{31}^{\theta_2} E_{42}^{\theta_3} E_{32}^{\theta_4}v \qquad \text{where\ } \theta_i\in\{0,1\}.
$$
Then $R=Jor(K_{\chi})$ is generated by $R_{11}=E_{42}E_{32}v$, $R_{22}=E_{31}E_{32}v$, $R_{12}=E_{32}v$ and $R_{21}=E_{31}E_{42}E_{32}v$.
If $E_{ij}$ $1\leq i,j\leq 2$ is the standard basis for $M_{1,1}^+$ we have the following action on $R$.
$$
\begin{array}{lll}
E_{ii}R_{jj}=\delta_{i,j} R_{jj} \qquad & & E_{kk}R_{ij}=\frac12R_{ij} \quad i,j,k=0,1\\
E_{12}R_{11}=\frac12(1-c)R_{12} \qquad & & E_{21}R_{11}=\frac12R_{21}\\
E_{12}R_{22}=\frac12(1+c)R_{12} \qquad & & E_{21}R_{22}=\frac12R_{21}-\frac12pR_{12}\\
E_{12}R_{12}=0 \qquad & & E_{21}R_{12}=\frac12R_{22}-R_{11}\\
E_{12}R_{21}=\frac12(1+c)R_{11}-\frac12(1-c)R_{22} \qquad & & E_{21}R_{21}=-\frac12pR_{11} 
\end{array}
$$
Rescaling, applying automorphism given by matrix  $\left[\begin{array}{cc}0&-1\\1&0\end{array}\right]$ which interchange action of $z_1$
and $z_{-1}$  we obtain the following action on $R^{op}$
$$
\begin{array}{lll}
E_{ii}R_{jj}=\delta_{i,j} R_{jj} \qquad & & E_{kk}R_{ij}=\frac12R_{ij}  \quad i,j,k=0,1\\
E_{12}R_{11}=\frac12R_{12} \qquad & & E_{21}R_{11}=\frac12R_{21}\\
E_{12}R_{22}=\frac12(1+c)R_{12}+\frac12kR_{21} \qquad & & E_{21}R_{22}=\frac12(1-c)R_{21}-\frac12pR_{12}\\
E_{12}R_{12}=-\frac12kR_{11} \qquad & & E_{21}R_{12}=\frac12R_{22}-(1-c)\frac12R_{11}\\
E_{12}R_{21}=\frac12(1+c)R_{11}-\frac12R_{22} \qquad & & E_{21}R_{21}=-\frac12pR_{11} 
\end{array}
$$
If $\chi=0$, $R$ is a regular representation of $M^{+}_{1,1}$. If $c=2$, $p=0=k$ then $Jor(S^2 V)=\langle R_{11}+R_{22}, R_{12}\rangle$ is a submodule in $R$, while $Jor(\Lambda^2 V)=R/Jor(S^2 V)$. We now can formulate the following

\begin{Cor} 
(a) If $\chi=(c,p,k)$ and $\chi\notin O_2$, then
    $\JJ^{\,\chi}$  has two up to isomorphism simple
    modules $R$ and $R^{op}$. 
    
(b) If $c=2,k=p=0$, then
    $\JJ^{\,\chi}$ has four up to isomorphism simple modules $Jor(S^2V)$, $Jor(\Lambda^2 V)$ and their opposite.  For an arbitrary $\chi\in O_2$ simple objects of $\JJ^{\,\chi}$ are obtained from those four by a suitable twist.
\end{Cor}

\subsection{Description of $\gmhh$}

\begin{lem}\label{auxKacspe} There are no non-trivial self-extensions of $V$ in the category of $\ssl(2|2)$-modules semisimple over $z_0$. 
\end{lem}
\begin{proof} See Lemma~\ref{selfext}.
  \end{proof}
\begin{thm}\label{half} Every block of $\Jhalf$ is equivalent to the category of finite-dimensional
  $\mathbb C[x,y]$-modules with nilpotent action of $x,y$, 
\end{thm}
\begin{proof} Theorem \ref{veryshortsl} implies that $\gmhh^{\chi}$ has two up to isomorphism simple object $L$ and $L^{op}$ and we may assume without loss of generality
  that $L=V$. Moreover, by Lemma \ref{weighargument} each block has one simple object. Thus, we may assume that this simple object is $V$. Let $\mathcal R=\mathbb C[[x,y]]$
  and $\mathcal I\subset\mathcal R$ be the maximal ideal. We will define $\mathcal R\otimes\hat\g$-module $\hat V$ such that for every $m$ the $\gh$-module
  $V^{(m)}:=\hat V/\mathcal I^m\hat V$
  is indecomposable of finite length with all simple subquotient isomorphic to $V$. Let $g(x,y)=\left[\begin{array}{cc}1&x\\ y&1+xy\end{array}\right]$ be an element of $SL(2,\mathcal R)$. Set $\hat V:=(\mathcal R\otimes V)^g$.
  By a straightforward computation we obtain that the action of $Z$ on $\hat V$ is given by the formulae:
  $$z_0\mapsto 1+2xy,\ z_1\mapsto -2x,\ z_{-1}\mapsto 2y(1+xy).$$
  This implies the desired properties of $\hat V$. We also see that $\hat V$ is a free rank $1$ module over $\mathcal R$ and that
  $z_0-1,z_1,z_{-1}$ act nilpotently on $V^{(m)}$ with the degree of nilpotency $m$. We claim that $V^{(m)}$ is projective in the category
  $F^m(\gmhh^{\chi})$ consisting of modules on which $(z-\chi(z))^m$ acts trivially.
  It suffices to show that every short exact sequence in $F^m(\gmhh^{\chi})$ of the form
  $$0\to V \to M\to V^{(m)}\to 0$$
  splits. Indeed, this sequence splits over $\mathcal R/\mathcal I^m$, and hence Lemma \ref{auxKacspe}
 implies splitting over $\gh$. Categories $\gmhh$ and $\Jhalf$ are equivalent therefore the statement follows. 
    \end{proof}

\subsection{Typical blocks}
We call $\chi$ typical if $K_{\chi}$ is simple or equivalently if  $\gm^{\,\chi}$ has two up to isomorphism simple modules $K_{\chi}$ and $K_{\chi}^{op}$ .
The condition that $\chi$ is typical is given by
$$c^2-kp\neq-4,\quad \chi\neq 0.$$

First, we assume that $\chi$ is semisimple and $p=k=0,c\neq 0$. We construct a certain deformation of $\hat K_{\chi}$ over the local ring
$\mathcal S:=\mathbb C[[x,y,t]]$. Our construction is similar to the one in the proof of Theorem \ref{half}. Let $\tilde K_{\chi}:=Ind^\g_\pp\mathbb C[[z_0-c-t]]$ and
$\hat K_{\chi}:=(\mathcal R\otimes\tilde K_{\chi})^g$ where $g$ is the same as in the proof of Theorem \ref{half}.
The action of $Z$ on $\hat K_{\chi}$ is given by the formula
\begin{equation}\label{eqnz1}
  z_0\mapsto (1+2xy)(c+t),\ z_1\mapsto -2x(c+t),\ z_{-1}\mapsto 2y(1+xy)(c+t).
\end{equation}
Let $\mathcal J$ denote the maximal ideal of $\mathcal S$ and $\hat K_{\chi}^{(m)}:=\hat K_{\chi}/\mathcal J^m$. Let  
$F^m(\gm^{\,\chi})$ denote the full subcategory of $\gm^{\,\chi}$ consisting of modules on which $(z-\chi(z))^m$ acts trivially.
\begin{lem}\label{auxKac} Assume $p=k=0$ and $c\neq 0$. Then there are no non-trivial self-extensions of $K_{\chi}$ in the category $F^1(\gm)$. 
\end{lem}
\begin{proof}  We need to show that $H^1(\gh,\gh_{\bar 0}; K_\chi^*\otimes K_\chi)$ vanishes. Since $K_\chi$ is the induced module, by the Shapiro Lemma
  it suffices to prove  $H^1(\pp,\pp_{\bar 0}; K_\chi)$. Write down the corresponding cochain complex:
  \begin{equation}\label{cochainm}
    0\to\Hom_{\g_{0}}(\mathbb C,K_\chi)=\mathbb C^2\xrightarrow{d_0} \Hom_{\g_{0}}(\mathbb \g_1,K_\chi)=\mathbb C^2\to \dots.
    \end{equation}
    Furthermore, $H^1(\pp,\pp_{\bar 0}; K_\chi)=\mathbb C$. Hence the image of $d_0$ is one dimensional. Modulo this image we can assume that our cocycle has
    the form $\varphi(x)=x^*v$ for all $x\in\g_1$, where $v$ is the highest weight vector. Let us write the cocycle condition
    $$x\varphi(x)=xx^*v=-[x,x^*]v=(c \operatorname{det} x) v=0.$$
    Clearly it does not hold for $c\neq 0$. Hence the statement.
  \end{proof}

  \begin{lem}\label{projtypsim} Let $k=p=0$ and $c\neq 0$. The module  $\hat K_{\chi}^{(m)}$ is projective in $F^m(\gm^{\,\chi})$ and
    $\operatorname{End}_{\gh}(\hat K_{\chi}^{(m)})\simeq \mathcal S/\mathcal J^m$.
  \end{lem}
  \begin{proof} For projectivity we note that an exact sequence in $F^m(\gm^{\,\chi})$ of the form
    $$0\to \hat K_{\chi}^{(m)}\to M\to \hat K_{\chi}\to 0$$
    splits over $\g_0\oplus Z$. On the other hand, Lemma \ref{auxKac} implies the splitting over $\gh$.
    The second assertion is a simple consequence of the fact that  $\dim\operatorname{End}_{\gh}(\hat K_{\chi}^{(m)})$ coincides with the length of 
    $K_{\chi}$ and hence equals $\dim\mathcal S/\mathcal J^m$.
    \end{proof}

\begin{thm}\label{extunitaltypicalsemisimple} Assume that $\chi$ is typical and semisimple. Then the category $\gm^{\,\chi}$ is a direct sum of two blocks,
  each block is equivalent to the category of
  finite dimensional modules
  over polynomial algebra $\mathbb C[x,y,t]$ with nilpotent action of $x,y,t$.
    \end{thm}
    \begin{proof} The first assertion is a  consequence of Lemma \ref{weighargument} and the second follows from Lemma \ref{projtypsim}.
    \end{proof}

    Now let us assume that $\chi$ is non-zero nilpotent. Without loss of generality we assume that $k=c=0$ and $p\neq 0$.

    \begin{lem}\label{auxKacnil} Assume $k=c=0$ and $p\neq 0$. Then there exist a unique up to isomorphism non-trivial self-extensions $\bar K_{\chi}$
      of $K_{\chi}$ in the    category $F^1(\gm)$. Moreover, $\bar K_{\chi}$ is projective in $F^1(\gm)$.  
    \end{lem}
    \begin{proof} Retain the notations of the proof of Lemma \ref{auxKac}. The argument with the cochain complex goes exactly as in this proof except the last step
      where we indeed obtain a non-trivial one-cocycle $\varphi(x)=x^*v$. Hence we have one non-trivial self-extension.

      For the second assertion we would like to show
      $$H^1(\gh,\gh_{\bar 0}; K_\chi^*\otimes \bar K_\chi)=H^1(\pp,\pp_{\bar 0}; \bar K_\chi)=0.$$
      From the long exact sequence we have an isomorphisms
      $$ H^0(\pp,\pp_{\bar 0};  K_\chi)\simeq\mathbb C\simeq H^0(\pp,\pp_{\bar 0};  \bar K_\chi),$$
      $$H^0(\pp,\pp_{\bar 0}; K_\chi)\simeq \mathbb C\simeq H^1(\pp,\pp_{\bar 0};K_\chi)$$
      and hence an injective map
      $$H^1(\pp,\pp_{\bar 0};\bar K_\chi)\to H^1(\pp,\pp_{\bar 0}; K_\chi).$$
      Consider $\gh_{\bar 0}\oplus \g_{-1}$ decomposition $\bar K_{\chi}=K_\chi\oplus K_\chi$. Then we may assume that the action of $\g_1$ is given by the formula
      $x(w,w')=(xw,\varphi(x)w+xw')$. Let $\psi\in\Hom_{\g_{0}}(\g_1,\bar K_\chi)$ be a 1-cocycle. We may assume that $\psi(x)=(x^*v,0)$. Then the cocycle
      condition $x\psi(x)=0$ becomes
      $$(xx^*v,(x^*)^2v)=(0,p \operatorname{det} x^*z_1v)=0.$$
      That implies $p=0$. Contradiction.
    \end{proof}
    We define a $\gh\otimes\mathbb C[[t]]$-module $T_\chi$ as follows:  $T_\chi=(K_\chi\oplus K_\chi)\otimes\mathbb C[[t]]$ 
    as a module over $\g_0\oplus\g_{-1}\oplus\mathbb Cz_0$ and define the action of $\g_1$ by
    $$x(u,w)=(xu+tx^*w,xw+x^*u)\ x\in\g_1, u,w\in K_\chi.$$
    Finally we set that $z_1$ acts as $pt$. It is straightforward that $T_\chi$ is indeed  a $\gh\otimes\mathbb C[[t]]$-module 
    and $T_\chi/t T_\chi$ is isomorphic to $\bar K_\chi$.
    
    Next, let $g=\left[\begin{array}{cc} (1+x)^{-1}& y\\0&1+x\end{array}\right]$ be an element of $SL(2,\mathcal R)$. Define $\mathcal S\otimes\gh$-modules 
    $Q_\chi$ and $Q_\chi^{(m)}$ by
    $$Q_{\chi}:=(\mathcal R\otimes T_{\chi})^g,\quad Q_{\chi}^{(m)}:=Q_{\chi}/\mathcal J^m.$$
    The action of $Z$ on $Q_\chi$  is given by
    \begin{equation}\label{eqnz1}
  z_0\mapsto (1+x)py,\quad z_1\mapsto -y^2p,\quad z_{-1}\mapsto pt+p(1+x)^2.
\end{equation}

\begin{lem}\label{projtypnil} The module  $Q_{\chi}^{(m)}$ is projective in $F^m(\gm^{\,\chi})$ and
    $$\operatorname{End}_{\gh}(Q_{\chi}^{(m)})\simeq (\mathcal S/\mathcal J^m)\otimes\mathbb C[\theta]/(\theta^2-t).$$
  \end{lem}
  \begin{proof} The proof of the first assertion is similar to the proof of Lemma \ref{projtypsim} with use of Lemma \ref{auxKacnil}. For the second,
    define action of $\theta$ on $Q_{\chi}^{(m)}$ by
    $\theta(u,w)=(tw,u)$. This defines a $\gh$-endomorphism of $Q_{\chi}^{(m)}$ satisfying $\theta^2=t$. The rest follows from comparison of dimensions.
  \end{proof}

  The following theorem is a consequence of the previous Lemma and Lemma \ref{weighargument}.
  \begin{thm}\label{typicalnil} Let $\chi$ be typical nilpotent, then $\gm^{\,\chi}$ (and thus  $\JJ^{\,\chi}$)  has two blocks, each of them is equivalent to the category
    of finite-dimensional $\mathbb C[x,y,\theta]$-modules with nilpotent action of $x,y,\theta$.
    \end{thm}

    \subsection{Geometry of $3$-parameter family of representations of $\hat\g$}
    We provide here a geometric construction which shades some light on the results of the previous subsection.      
      We will construct a three-dimensional family of representation of $\hat\g$. We have
      $$\g_{\bar 1}=U\times \mathbb C^2,$$
      where $U$ is the $4$-dimensional irreducible representation of $\g_{\bar 0}=\mathfrak{sl}(2)\oplus\mathfrak{sl}(2)$ with highest weight $(1,1)$. For every line $\ell\subset\mathbb C^2$, we have
      a commutative subalgebra $\g_{\ell}\subset\g_{\bar 1}$, and it can be lifted to the subalgebra $\hat\g_{\ell}$ with one-dimensional center $Z_\ell\subset Z$.
      Note that $Z_\ell$ is a line $\mathbb C^3=Z$, thus, we have
      the map $\psi:\mathbb P^1\to \mathbb P(Z)\simeq\mathbb P^2$. Now let $\chi\in\mathbb Z^*$, we say that $\ell$ is $\chi$-compatible if
      $\chi([\g_\ell,\g_\ell])=\chi(\psi)=0$. To compute $\psi$ consider the realization
      $$\g_\ell=\left\{X_B=\left[\begin{array}{cc} 0&t_1B\\t_2B^*&0\end{array}\right]\right\}$$
      where $(t_1,t_2)$ are homogeneous coordinates of $\ell$. Then
      $$[X_B,X_B]=\det B( t_1^2z_1+2t_1t_2z_0+t_2^2z_{-1}).$$
      Thus, $\psi$ is the Veronese map. Therefore for every $\chi\neq 0$ there exists at most two choices of a compatible $\ell$. More precisely, for a semisimple
      $\chi$ we have two $\chi$-compatible lines, and for a nilpotent $\chi$ a $\chi$-compatible $\ell$ is unique. Let
      $$M_{\chi}:=\operatorname{Ind}^{\hat\g}_{\hat\g_{\bar 0}+\g_\ell}\mathbb C_{\chi}.$$
      If $k=0$ then $M_\chi$ is isomorphic to $K_\chi$.
      Let $$\mathcal M=\{(\chi,\ell)\,|\,\chi\neq 0,\chi(\psi(\ell))=0\}$$ with obvious structure of smooth complex manifold.
      By construction $\mathcal M$ is isomorphic to a non-trivial $SL(2)$-equivariant two-dimensional vector bundle on $\mathbb P^1$.
      Our construction defines a vector bundle on $\mathcal M$ with fiber isomorphic
      to $M_\chi$. For every open set $\mathcal U\subset \mathcal M$, we thus obtain a representation of the Lie superalgebra
      $\mathcal O(\mathcal U)\otimes \hat\g$. For every point $(\chi,\ell)\in \mathcal M$ we obtain a representation of
      $\mathcal O_{\chi,\ell}\otimes \g$, where $\mathcal O_{\chi,\ell}$ is the local ring of the point. If $\mathcal J_{\chi,\ell}$ denote the unique maximal
      ideal of $\mathcal O_{\chi,\ell}$, the quotient $\mathcal O_{\chi,\ell}/\mathcal J^m_{\chi,\ell}$ is isomorphic to $\mathbb C[x_1,x_2,x_3]/(x_1,x_2,x_3)^m$.
      In the previous section we have proved  that for a non-zero semisimple $\chi$ the $\gh$-module
      $$M^{(m)}_\chi\otimes_{\mathcal O_{\chi,\ell}}\mathcal O_{\chi,\ell}/\mathcal J^m_{\chi,\ell}$$
      is projective in $F^{(m)}(\gm)$.
                                                        
      \subsection{Atypical blocks}
      We proceed to the description of $\gm^{\,\chi}$ in the case of an atypical $\chi$.
      This amounts to considering two cases $k=p=0,c=2$ and $\chi=0$.
      We start with the first case.
      \begin{lem}\label{atypicalsemisimple} Let $k=p=0,c=2$. There is the following non-split exact sequence
        $$0\to S^2V\to K_\chi\to \Lambda^2V\to 0.$$
      \end{lem}
      \begin{proof} The map $\mathbb C_{\chi}\to\Lambda^2V_0\to \Lambda^2V$ is a homomorphism of $\pp$-modules.
        Hence by Frobenius reciprocity we have a surjection $K_\chi\to \Lambda^2V$. On the other hand, 
        $K_{\chi}\simeq\operatorname{Coind}^\g_\pp (\mathbb C_{\chi})$  and 
       $S^2V\to S^2V_1\to \mathbb C_\chi$ is an homomorphism of $\pp$-modules. Hence we have an injection $S^2V\to K_{\chi}$. 
       Finally, $K_\chi^{\g_1}=\mathbb C_\chi$ which implies indecomposability of $K_\chi$.
      \end{proof}

      By Lemma \ref{weighargument} we obtain that $\gm^{\,\chi}$ has two blocks obtained from each other by parity switch. By Lemma \ref{projtypsim} $\hat K_{\chi}^{(m)}$
      is a projective cover of $\Lambda^2V$ in $F^m(\gm^{\,\chi})$. To construct a projective cover of $S^2V$ consider the automorphism $\pi$ of $\gh$
      defined by $\pi\left[{}^A_C{}^B_D\right]=\left[{}^D_B{}^C_A\right]$, $\pi(z_0)=z_0$, $\pi(z_{\pm 1})=z_{\mp 1}$. We have $V^\pi\simeq V^{op}$ and hence
      $(\Lambda^2V)^\pi\simeq S^2V$. Thus, $(\hat K_{\chi}^{(m)})^\pi$ is a projective cover of $S^2V$ in  $F^m(\gm^{\,\chi})$.
      The algebra $\operatorname{End}_{\gh}(\hat K_{\chi}^{(m)}\oplus(\hat K_{\chi}^{(m)})^\pi)$ is isomorphic to the path algebra of the quiver 
     
      $$
      Q\qquad \xymatrix{\bullet \ar@(ul,ur)[]|{\alpha} \ar@(dl,dr)[]|{\zeta}\ar@/^0.4pc/[r]^\beta & \bullet\ar@(ul,ur)[]|{\gamma}\ar@(dl,dr)[]|{\eta}\ar@/^0.4pc/[l]^{\delta}}  \qquad \text{with relations} \quad R=\left\{\begin{array}{c} 
       \beta\alpha=\gamma\beta,\ \beta\zeta=\eta\beta,\ \zeta\delta=\delta\eta\\
       \alpha\delta=\delta\gamma,\ \alpha\zeta=\zeta\alpha,\ \gamma\eta=\eta\gamma\end{array}\right\}
       $$     
    \\
      Therefore we obtain the following
      \begin{thm}\label{atypical1} Let $\chi$ be semisimple atypical. Each of two blocks of $\gm^{\,\chi}$ (and  $\JJ^{\,\chi}$) is equivalent to the category of finite-dimensional
        nilpotent representations of the quiver $Q$ with relations $R$.
      \end{thm}
     \noindent  Observe that the algebra obtained in Theorem~\ref{description-JP(2)} is a quotient of $(Q,R)$. Hence $(Q,R)$ has wild representation type.

      Now let us consider the case $\chi=0$. We start by describing the projective cover of $\ad$ in $\ggm$. Recall that $\g=\mathfrak{psl}(2|2)$.
      We set $\g^+:=\g_0\oplus\g_1$. Consider the $\g^+$-module $S:=\g_1\oplus \C$ with action of $x\in\g_1$ given by $x(y,1)=(0,\operatorname{tr}(xy))$.

      \begin{lem}\label{auxilaryzero} $\Ext^1_{\g^+}(S,\mathbb C)=\Ext^1_{\g^+}(S,\ad)=0$.
      \end{lem}
      \begin{proof} A simple computation shows that
        $$\Ext^1_{\g^+}(\g_1,\mathbb C)=H^1(\g^+,\g_0; \g_1)=\mathbb C,$$
        $$\Ext^1_{\g^+}(\mathbb C,\mathbb C)=H^1(\g^+,\g_0; \mathbb C)=0.$$
        Using the long exact sequence associated with the short exact sequence of $\g^+$-modules $0\to\mathbb C\to S\to\g_1\to 0$ we get
        $$0\to \Hom_{\g^+}(\mathbb C,\mathbb C)\to \Ext^1_{\g^+}(\g_1,\mathbb C)\to \Ext^1_{\g^+}(S,\mathbb C)\to 0,$$
        which implies  $\Ext^1_{\g^+}(S,\mathbb C)=0$.

        To prove the second vanishing we note that $K_0$ is both injective and projective in the category of $\g^+$-modules.
        Let $K'_0$ be the submodule defined the exact sequence $0\to K'_0\to K_0\to\mathbb C\to 0$. Since $\Hom_{\g^+}(S,\mathbb C)=0$
        and $\Ext^1_{\g^+}(S,K_0)$, we obtain $\Ext^1_{\g^+}(S,K'_0)=0$. Next we consider the exact sequence
        $$0\to\mathbb C\to K'_0\to\ad\to 0.$$
        Form the corresponding long exact sequence we have an embedding $\Ext^1_{\g^+}(S,\ad)\to \Ext^2_{\g^+}(S,\mathbb C)$. We will show that
        $\Ext^2_{\g^+}(S,\mathbb C)=H^2(\g^+,\g_0; S^*)=0$. Indeed, we have
        $$\Hom_{\g_0}(\g_1\otimes S,\mathbb C)=\Hom_{\g_0}(\Lambda^2\g_1\otimes S,\mathbb C)=\mathbb C.$$
        On the other hand $H^1(\g^+,\g_0; S^*)=\Ext^1_{\g^+}(S,\mathbb C)=0$, therefore the differential
        $$d:\Hom_{\g_0}(\g_1\otimes S,\mathbb C)\to\Hom_{\g_0}(\Lambda^2\g_1\otimes S,\mathbb C)$$
        is an isomorphism and there are no non-trivial two cocycles.
        The proof of lemma is complete.
      \end{proof}

      Let $P$ be the maximal quotient of $\operatorname{Ind}^\g_{\g^+}(S)$ which lies
      in $\ggm$. By the Shapiro lemma we have 
      $$\Ext^1_{\g}(\operatorname{Ind}^\g_{\g^+}(S),\ad)=\Ext^1_{\g}(\operatorname{Ind}^\g_{\g^+}(S),\mathbb C)=0.$$
      If $N$ is the kernel of the canonical projection $\operatorname{Ind}^\g_{\g^+}(S)\to P$, then $\Hom_{\g}(N,\ad)=\Hom_{\g}(N,\mathbb C)=0$
      and hence $\Ext^1_{\g}(P,\ad)=\Ext^1_{\g}(P,\mathbb C)=0$. Thus, $P$ is projective in $\ggm$. Furthermore, it is not difficult to see that
      $N$ is generated by a highest weight vector of weight $(2,2)$ and the structure of $P$ can be described by the exacts sequence
      $$0\to\mathbb C^3\to P\to \ad\to 0.$$

      Next we define $P^{(m)}$ as the maximal quotient of the induced module $\operatorname{Ind}^{\hat \g}_{\pp}(S\otimes (S(Z)/(Z)^m))$. Repeating the argument of the
      proof of Lemma \ref{projtypsim} one can show that $P^{(m)}$ is projective in $F^m(\gm^{0})$. It is always straightforward  $S(Z)/(Z)^m$ is isomorphic
      to $\operatorname{End}_{\gh}(P^{(m)})$. Finally $Jor(P^{(m)})$ is projective in $F^m(\JJ^0)$ and we obtain the following
      \begin{thm}\label{lastm} The category $\JJ^{\,0}$ is equivalent to the category of finite-dimensional representations of the polynomial ring $\mathbb C[x,y,t]$
        with nilpotent action of $x,y,t$.
        \end{thm}      
       
        \section{Jordan superalgebra of a bilinear form}
        Let $V=V_{\bar 0}+V_{\bar 1}$ be a $\Z_2$-graded vector space equipped with a 
         bilinear form $(\cdot |\cdot )\,:\, V\times V\to\C$ which is symmetric on
         $V_{\bar 0}$, skewsymmetric on $V_{\bar 1}$ and satisfies
        $(V_{\bar 0}|V_{\bar 1})=0=(V_{\bar 1}|V_{\bar 0})$. Then superspace $J=\C 1\oplus V$, where $1\in J_0$
        has a Jordan superalgebra structure with respect to a product 
        $$
        (\alpha1+a)\cdot(\beta1+b)=(\alpha\beta+(a|b))1+\alpha b+\beta a, \quad \alpha,\,\beta\in\C,\ a,b\in V.
        $$
        Moreover if $(\cdot|\cdot)$ is non-degenerate then $J$ is simple. Let $\dim V_{\bar 0}=m-3$, $\dim V_{\bar 1}=2n$
        then the TKK construction of $J$ gives the orthosymplectic Lie superalgebra 
        $$
        \mathfrak{osp}(m|2n)=\left\{ A\in\gl(m|2n)\,|\, (Ax,y)+(-1)^{|A||x|}(x,Ay)=0,\ x,y\in V\right\}.
        $$
        Denote $\g=\mathfrak{osp}(m|2n)$ with $m\geq 3$ and $n\geq 1$.  
        In what follows we need the description of the roots of $\g$ 
        $$
        \begin{array}{c}
        \Delta_{\bar 0}=\{\pm (\varepsilon_i\pm\varepsilon_j)\mid 1\leq i<j\leq k\}\cup \{\pm (\delta_i\pm\delta_j)\mid 1\leq i<j\leq n\},\\
        \Delta_{\bar 1}=\{\pm (\varepsilon_i\pm\delta_j)\mid 1\leq i\leq k, 1\leq j\leq n\}\end{array} \qquad \text{if}\  m=2k
        \ \text {is\ even}$$
     and
        $$
        \begin{array}{c}\Delta_{\bar 0}=\{\pm (\varepsilon_i\pm\varepsilon_j),\pm \varepsilon_i\mid 1\leq i<j\leq k\}\cup \{\pm (\delta_i\pm\delta_j)\mid 1\leq i<j\leq n\},\\
\Delta_{\bar 1}=\{\pm (\varepsilon_i\pm\delta_j),\pm\delta_j\mid 1\leq i\leq k, 1\leq j\leq n\}\end{array}
\qquad \text{if}\  m=2k+1
        \ \text {is\ odd.}$$        

The semisimple element which defines the short grading on $\g$  is $h:=\varepsilon_1^\vee$.  The short $\ssl(2)$-subalgebra 
is spanned by $h$ and $e,f$. The definition of $e,f$ depends on the parity of $m$. If $m=2k+1$
$e\in \g_{\varepsilon_1},\, f\in \g_{\varepsilon_1}$ are roots vector corresponding to the short roots, For $m=2k$ let $\alpha=\varepsilon_1-\varepsilon_2,\ \beta=\varepsilon_1+\varepsilon_2 $ and
$e\in\g_{\alpha}\oplus\g_{\beta}$, $f\in\g_{-\alpha}\oplus\g_{-\beta}$. In both cases the short grading $\g=\g[-1]\oplus\g[0]\oplus\g[1]$ satisfies the condition
$\g_\gamma\in\g[i]$ iff $(\gamma,\varepsilon_1)=i$. We set $J:=Jor(\g)$.

\subsection{Modules in $\ggm$}
We choose the Borel subalgebra of $\g$ associated with the set of simple roots
$$\delta_1-\delta_2,\dots,\,\delta_{n-1}-\delta_n,\,\delta_n-\varepsilon_1,\,
\varepsilon_1-\varepsilon_2,\dots,\varepsilon_{k-1}-\varepsilon_k,\,\varepsilon_{k-1}+\varepsilon_k \quad \text{for\ } m=2k$$
and
$$\delta_1-\delta_2,\dots,\delta_{n-1}-\delta_n,\,\delta_n-\varepsilon_1,\,\varepsilon_1-\varepsilon_2,\dots,\varepsilon_{k-1}-\varepsilon_k,\,\varepsilon_k \quad \text{for\ } m=2k+1.$$
Denote by $L(\lambda)$ the simple $\g$-module with highest weight $\lambda$ with respect to this Borel subalgebra.
The invariant bilinear form on $\g$ induces the form on $\h$ and $\h^*$, the latter is defined in $\varepsilon,\delta$-basis by
$$(\varepsilon_i,\varepsilon_j)=\delta_{i,j},\,(\delta_i,\delta_j)=-\delta_{i,j},\, (\varepsilon_i,\delta_j)=0.$$
For $\mu\in\h^*$ such that $(\mu,\mu)\neq 0$ we define $\mu^\vee\in \h$ satisfying $\nu(\mu^\vee)=\frac{2(\mu,\nu)}{(\mu,\mu)}$.
The Casimir element $\Omega\in U(\g)$ is defined by the invariant form acts on $L(\lambda)$ by the scalar $(\lambda+2\rho,\lambda)$ where
$$\rho=\frac{1}{2}\sum_{\alpha\in\Delta_{\bar 0}}\alpha-\frac{1}{2}\sum_{\alpha\in\Delta_{\bar 1}}\alpha.$$
It was shown in \cite{Kac2} that $\hat{\g}=\g$.

According to \cite{ZM2} the Jordan superalgebra $J$ does not have finite-dimensional one sided modules due to the fact that the universal enveloping of $J$ is
the tensor product of the Clifford and Weyl algebras. Thus, $\gmh$ is empty.
The classification of simple objects of $\ggm$ is done in \cite{ZM}. We give the proof using TKK here for the sake of completeness.
\begin{lem}\label{simpbil} A simple finite-dimensional $\g$-module $L(\lambda)$ lies in $\ggm$ if and only if $\lambda=a\delta_1$ for $a\in\mathbb Z_{\geq 0}$.  In this case $L(\lambda)$ is isomorphic to $\Lambda^a(V)$ where $V$ is the standard $\g$-module.
\end{lem}
\begin{proof} Let $\lambda=\sum_{j=1}^na_i\delta_i+\sum_{i=1}^kb_i\varepsilon_i$. Since $L(\lambda)$ is finite-dimensional we have by the
  dominance condition
  $$a_1\geq\dots\geq a_n\geq 0,\, a_i\in\mathbb Z,$$
  $$ b_i\in\mathbb Z/2,\, b_1\geq\dots \geq b_k\geq 0\ \text{if}\ m=2k+1,$$
  $$ b_i\in\mathbb Z/2,\, b_1\geq\dots\geq |b_k|\ \text{if}\ m=2k,$$
  and finally if $l$ is the maximal index for which $b_l\neq 0$ we have $a_n\geq l$.
  On the other hand, since $L(\lambda)$ has a short grading, we have $b_1=(\lambda,\varepsilon_1)=0$ or $1$.

  First, assume that $b_1=1$. Consider the odd simple root $\alpha=\delta_n-\varepsilon_1$, then $\lambda-\alpha$ is not a weight of $L(\lambda)$.
  That is possible only if $(\lambda,\alpha)=0$. But $(\lambda,\alpha)=a_n+b_1>0$. A contradiction.

  Therefore, $b_1=0$. Hence $\lambda=\sum_{i=1}^n a_i\delta_i$. To finish the proof we compute the highest weight of $L(\lambda)$ with respect to the Borel subalgebra
  obtained from our Borel subalgebra by the reflections with respect to the isotropic roots $\delta_n-\varepsilon_1,\dots,\delta_1-\varepsilon_1$. Recall the formula
  $$r_{\alpha}(\mu)=\begin{cases}\mu-\alpha\,\,\,\text{if}\,\,(\mu,\alpha)\neq 0, \\ \mu\,\,\,\text{if}\,\,(\mu,\alpha)=0. \end{cases}$$
    Thus, we have
    $$\mu:=r_{\delta_1-\varepsilon_1}\dots r_{\delta_n-\varepsilon_1}(\lambda)=\lambda+l\varepsilon_1-\sum_{j=1}^l\delta_i,$$
    where $l$ is the maximal index such that $a_l\neq 0$. Since $(\mu,\varepsilon_1)=\pm 1,0$ we obtain $l=1$ or $l=0$.
    Therefore $\lambda=a\delta_1$.
    That proves the first assertion. The second assertion is straightforward.
  \end{proof}

  \begin{thm}\label{semisimplebil} The category $\ggm$ is semisimple. Hence the category $\JJ$ is semisimple.
  \end{thm}
  \begin{proof} We have to show that
    \begin{equation}\label{Casimir}
      \Ext^1(L(a\delta_1), L(b\delta_1))=0.
      \end{equation}
    First we note that if  $\Ext^1(L(a\delta_1), L(b\delta_1))\neq 0$
    then the Casimir element acts on both modules by the same scalar. In our case it amounts to the condition
    $$a(a+2n-m)=b(b+2n-m).$$
    Since both $a,b$ are non-negative integers this is only possible if $a+b=m-2n$. All modules in question are self-dual it suffices to prove (\ref{Casimir})
    in the case when $b>a$ or equivalently
    $$H^1(\g,\g_{\bar 0}; \Lambda^a V\otimes \Lambda^b V)=0.$$
    We have the decomposition
\begin{equation}\label{decomp}
  \Lambda^c(V)=\bigoplus_{p+q=c}S^p(V_{\bar 1})\otimes\Lambda^q(V_{\bar 0}).
  \end{equation}
    The highest weight vector $v$ of $\Lambda^a(V)$ lies in the component $S^a(V_{\bar 1})$.
    We claim that if $\varphi\in\Hom_{\g_{\bar 0}}(\g_{\bar 1}\otimes \Lambda^a(V),\Lambda^b(V))$ is a non-trivial cocycle
    then $\varphi(g_{\bar 1},v)\neq 0$. Indeed, assume the opposite. Consider the sequence $0\to L(b\delta_1)\to M\to L(a\delta_1)\to 0$ defined by the
    cocycle $\varphi$. The $\g$-submodule
    of $M$ generated by $v$ is isomorphic to $L(a\delta_1)$ and the sequence splits.
    Thus, if there is a non-trivial extension we must have $\Hom_{\g_{\bar 0}}(\g_{\bar 1}\otimes S^a(V_{\bar 1}),\Lambda^b(V))\neq 0$. Furthermore,
    $\g_{\bar 1}\simeq V_{\bar 1}\otimes V_{\bar 0}$ as
    a $\g_{\bar{0}}$-module, therefore (\ref{decomp}) implies that $\Lambda^b(V)$ must have a component isomorphic to
    $S^{a+1}(V_{\bar 1})\otimes V_{\bar 0}$ or to  $S^{a-1}(V_{\bar 1})\otimes V_{\bar 0}$.
    This is possible only if $b=a+2$, $b=a+1+m$, $b=a$ or $b=a-1+m$. The case $b=a$ can be dismissed right away since there is no self-extension.
    The condition (\ref{Casimir}) helps to exclude the cases $b=a+1+m$, $b=a-1+m$. The following lemma completes the proof.
    \begin{lem} $$\Ext^1(\Lambda^a V,\Lambda^{a+2}V)=0.$$
    \end{lem}

    \begin{proof} We will show that there is no cocycle $\varphi\in\Hom_{\g_{\bar 0}}(\g_{\bar 1}\otimes \Lambda^a(V),\Lambda^b(V))$. Consider the restriction
      $\varphi: \g_{\bar 1}\otimes S^a(V_{\bar 1})\to S^{a+1}(V_{\bar 1})\otimes V_{\bar 0}$. Let $X_{u\otimes w}\in \g_{\bar 1}$ be the element corresponding to
      $u\otimes w$ for $u\in V_{\bar 1}$ and
      $w\in V_{\bar 0}$. Then without loss of generality we may assume
      $$\varphi(X_{u\otimes w},x)=u\wedge w\wedge x.$$
      In the case when $X_{u\otimes w}$ belongs to the Borel subalgebra and $x=v$ is a highest weight vector of $\Lambda^a(V)$ the cocycle condition implies
      $$X_{u\otimes w}\varphi(X_{u\otimes w},v)=X_{u\otimes w}(u\wedge w\wedge v)=0.$$
      Since $X_{u\otimes w}v=0$, the above condition actually implies $X_{u\otimes w}(u\wedge w)=0$. Now we use the formula
      $$X_{u\otimes w}(u\wedge w)=(w|w)u\wedge u. $$
      Let $u$ be a weight vector of weight $\delta_1$ and $w=w'+w''$ where $w',w''$
      are weight vector of weights $\varepsilon_1$ and $-\varepsilon_1$ respectively. Then $X_{u\otimes w}$ is a sum of root vectors in $\g_{\delta_1+\varepsilon_1}$ 
and $\g_{\delta_1-\varepsilon_1}$, hence $X_{u\otimes w}$ belongs to the Borel subalgebra. But $(w|w)\neq 0$. Thus we obtain a contradiction with the cocycle condition.
    \end{proof}
    
  \end{proof}

  \section{Acknowledgement} The first author was supported by Fapesp grant FAPESP 2017/25777-9. The second author was supported by  NSF grant 1701532.
  The first author express her gratitude to Department of Mathematics, University of California, Berkeley, were most of the work was done.


\begin{thebibliography}{9}

\bibitem{Tits} J.~Tits, {\it Une classe d'alg\'ebres de Lie en relation avec les alg\'ebres de Jordan,\/}
Indag. Math. {\bf 24}, (1962), 530-534.

\bibitem{Kan2} I.~L.~Kantor, {\it Classification of irreducible transitively differential groups, \/}
Soviet Math. Dokl., {\bf 5}, (1964), 1404-1407.

\bibitem{Koecher} M.~Koecher, {\it Imbedding of Jordan algebras into Lie algebras I,\/}
Amer. J. Math. {\bf 89}, (1967), 787-816.  

\bibitem{Kac1} V.~Kac, {\it Classification of simple $\Z$-graded Lie superalgebras and simple Jordan superagebras,\/}
Comm. Algebra, {\bf 5}, no.13, (1977), 1375-1400.

\bibitem{Kac2} V.~ G.~Kac, {\it Lie superalgebras,\/}  Adv. Math. {\bf 26} (1977), 8-96.

\bibitem{Kap} I.~Kaplansky, {\it Superalgebras,\/} Pacific J.~Math., {\bf 86}, no.1, (1980), 93-98.

\bibitem{Kan} I.~L.~Kantor, {\it Jordan and Lie superalgebras determined by a Poisson algebra,\/}
Amer. Math. Soc. Transl. Ser. 2, 151, Amer. Math. Soc., Providence, RI, 1992.

\bibitem{J} N.~Jacobson, {\it Structure and representations of Jordan
algebras,\/}  AMS Colloq.~Publ. {\bf 39}, AMS, Providence 1968.

\bibitem{Sh1} A.~S.~Shtern, {\it Representations of an exceptional Jordan superalgebra,\/}
Funct.Anal Appl {\bf 21}, no.3, (1987), 253-254.

\bibitem{Sh2} A.~S.~Shtern, {\it Representations of finite dimensional Jordan superalgebras of Poisson bracket,\/} Comm. Algebra, {\bf 23}, (1995), no. 5, 1815–1823.


\bibitem{ZM} C.~Mart\'\i nez, E.~Zelmanov, {\it Representation theory of Jordan superalgebras I,\/}
Trans. AMS, {\bf 362}, (2010) no.2, 815–846.

\bibitem{ZM2} C.~Mart\'\i nez, E.~Zelmanov, {\it Specialization of Jordan Superalgebras,\/}
Canad. Math. Bull., {\bf 45}, (2002) no.4, 653–671.

\bibitem{ZM3} C.~Mart\'\i nez, E.~Zelmanov {\it Unital bimodules over simple Jordan superalgebras $D(t)$,\/}
Trans. AMS, {\bf 358}, (2006), 3637-3649.

\bibitem{ZM4} C.~Mart\'\i nez, E.~Zelmanov, {\it Jordan Superalgebras and their Representations,\/} 
Contemporary Mathematics, {\bf 483}, (2009), 179–194.

\bibitem{MSZ} C.~Mart\'\i nez, I.~Shestakov, E.~Zelmanov, {\it Jordan bimodules over the superalgebras $P(n)$ and $Q(n)$,\/}
Trans. AMS, {\bf 362}, (2010) no.4, 2037-2051.

\bibitem{ShO} O.~F.~Solarte, I.~Shestakov, {\it Irreductible representations of the simple Jordan 
superalgebra of Grassmann Poisson bracket,\/} J.~Algebra, {\bf 455}, (2016), 291-313.

\bibitem{Tr} M.~Trushina, {\it Modular representations of the Jordan superalgebras $D(t)$,\/}
J.~Algebra, {\bf 320}, (2008), 1327-1343.

\bibitem{CanK} N.~Cantarini, V.~Kac, {\it Classification of linearly
compact simple Jordan and generalized Poisson superalgebras,\/}
J.~Algebra, {\bf 313}, (2007), 100-124.

\bibitem{KMC} D.~King, K.~Mccrimmon, {\it The Kantor construction of Jordan superalgebras,\/}
Comm. Algebra, {\bf 20}, no.1, (1992), 109-126.

\bibitem{Sha1} A.~V.~Shapovalov,  {\it Finite-dimensional irreducible representations of Hamiltonian Lie superalgebras,\/} 
Mat. Sbornik, {\bf 107}, (1978), 259-274.

\bibitem{Sha2} A.~V.~Shapovalov, {\it Invariant differential operators and irreducible representations of finite-dimensional 
Hamiltonian and Poisson Lie superalgebras,\/} Serdica Bulgar. Math. Publ.{\bf 7} (1981), 337-342.

\bibitem{KS} I.~Kashuba, V.~Serganova, {\it The Kantor-Koecher-Tits construction,\/}  J.~ Algebra, {\bf 481},
420-463.
(2017), 

\bibitem{KS2} I.~Kashuba, V.~Serganova, {\it  Special modules for Jordan algebras,\/}  in preparation.

\bibitem{Gelf-Manin} S.~Gelfand, Yu.~Manin, {\it Homological Algebra,\/}
Springer-Verlag, Berlin, (1999).

\bibitem{Gab} P.~Gabriel, {\it Indecomposable Representations II,\/}
Symposia Mathematica, vol. XI Academic Press, London (1973), 81–104.

\bibitem{Gab2} P.~Gabriel, {\it Auslander–Reiten sequences and representation-finite algebras,\/} in:
Representation Theory, I, Proc. Workshop, Carleton Univ., Ottawa, ON, 1979, in: Lecture Notes in Math., {\bf 831}, 
Springer, Berlin, (1980), 1–71.

\bibitem{Vera} V.~Serganova, {\it  Representations of a central extension of the simple Lie superalgebra ${\mathfrak p}(3)$},
Sao Paulo J. of Math Science, {\bf 12},  (2018), 359-376.

\bibitem{Han} Y.~Han, {\it Wild two-point algebras,\/} J.~Algebra, {\bf 247}, (2002), 57-77.

\bibitem{F} D. Fuchs, {\it Cohomology of infinite-dimensional Lie algebra,\/} Springer, Boston, (1987). 

\bibitem{Vera2} V.~V.~Serganova, {\it Automorphism of simple Lie superalgebras,\/} Izv. Akad. Nauk SSSR
Ser. Mat., {\bf 48}(3), (1984), 585-598.

\bibitem{MS} C.~Mart\'\i nez, I.~Shestakov, {\it Jordan bimodules over the superalgebra $M_{1|1}$,\/}  Glasgow Math. ~J., (2019).


\end{thebibliography}
\end{document}